\documentclass[reqno]{amsart}
\usepackage{amsmath}
\usepackage{amscd,amsthm,amsfonts,amsopn,amssymb}
\usepackage{hyperref}
\usepackage{epsfig}
\usepackage{graphicx,verbatim}
\numberwithin{equation}{section}

\newtheorem{theorem}{Theorem}[section]
\newtheorem{proposition}{Proposition}[section]

\newtheorem{definition}{Definition}[section]
\newtheorem{lemma}{Lemma}[section]

\theoremstyle{remark}

\DeclareMathOperator{\supp}{supp\,}
\DeclareMathOperator{\dist}{dist\,}

\long\def\symbolfootnote[#1]#2{\begingroup
\def\thefootnote{\fnsymbol{footnote}}\footnote[#1]{#2}\endgroup} 

\allowdisplaybreaks

\begin{document}

\title[Defocusing energy-supercritical NLW in dimension 5]{The Defocusing Energy-supercritical Cubic Nonlinear Wave Equation in Dimension Five}

\author{Aynur Bulut}

\address{School of Mathematics, Institute for Advanced Study, Princeton, NJ 08540}

\email{abulut@math.ias.edu}
\thanks{2010 Mathematics Subject Classification: 35L71, 35B44, 35P25}

\begin{abstract}
We consider the energy-supercritical nonlinear wave equation $u_{tt}-\Delta u+|u|^2u=0$ with defocusing cubic nonlinearity in  dimension $d=5$ with no radial assumption on the initial data.  We prove that a uniform-in-time {\it a priori} bound on the critical norm implies that solutions exist globally in time and scatter at infinity in both time directions.  Together with our earlier works in dimensions $d\geq 6$ with general data \cite{BulutCubic} and dimension $d=5$ with radial data \cite{BulutRadial}, the present work completes the study of global well-posedness and scattering in the energy-supercritical regime for the cubic nonlinearity under the assumption of uniform-in-time control over the critical norm.
\end{abstract}

\maketitle

\tableofcontents

\setlength{\parskip}{0.1in}

\section{Introduction}
We consider the initial value problem for the nonlinear wave equation with defocusing cubic nonlinearity in the energy supercritical regime, that is, dimensions $d\geq 5$,
\begin{align*}
\textrm{(IVP)}\quad\left\lbrace\begin{array}{rl}\partial_{tt}u-\Delta u+|u|^2u&=0,\\
(u,\partial_t u)|_{t=0}&=(u_0,u_1)\in \dot{H}_x^{s_c}\times\dot{H}_x^{s_c-1}
\end{array}\right.
\end{align*}
where $u:I\times\mathbb{R}^d\rightarrow\mathbb{R}$ with $0\in I\subset \mathbb{R}$ a time interval.

The scaling symmetry
\begin{align*}
u(t,x)\mapsto u_{\lambda}(t,x):=\lambda u(\lambda t,\lambda x).
\end{align*}
maps the set of solutions of IVP to itself and, moreover, the $\dot{H}_x^{s_c}\times\dot{H}_x^{s_c-1}$ norm is preserved by the scaling, where we identify the {\it critical regularity} as $s_c=\frac{d-2}{2}$.

We recall that solutions to (IVP) conserve the {\it energy}
\begin{align*}
E(u(t),u_t(t))=\int_{\mathbb{R}^d}\frac{1}{2}|u_t(t,x)|^2+\frac{1}{2}|\nabla u(t,x)|^2+\frac{1}{4}|u(t,x)|^4dx=E(u(0)),
\end{align*}
which is both finite and left invariant by the scaling when $s_c=1$.  In this case, we refer to (IVP) as the {\it energy-critical} nonlinear wave equation.  The range $s_c>1$, that is $d\geq 5$, is therefore known as the {\it energy-supercritical regime} for NLW.

In a recent series of works \cite{BulutCubic,BulutRadial} treating dimensions $d\geq 6$ with general (possibly non-radial) initial data, and dimension $d=5$ with radial initial data, we proved that any solution to (IVP) which satisfies an {\it a priori} uniform-in-time bound over the critical norm must exist globally in time and scatter; see also \cite{BulutThesis} and \cite{BulutContemporary}.  The goal of the current paper is to treat the remaining case of dimension $d=5$ with non-radial initial data,  completing the analysis of global well-posedness and scattering under the assumed a priori bound for (IVP) in the energy-supercritical regime.

More precisely, we consider {\it strong solutions} to
\begin{align*}
\textrm{(NLW)}\quad\left\lbrace\begin{array}{rl}\partial_{tt}u-\Delta u+|u|^2u&=0,\\
(u,\partial_t u)|_{t=0}&=(u_0,u_1)\in \dot{H}_x^{3/2}\times\dot{H}_x^{1/2},\end{array}\right.
\end{align*}
that is, functions $u:I\times\mathbb{R}^5\rightarrow\mathbb{R}$ such that for every $K\subset I$ compact, $(u,u_t)\in C_t(K;\dot{H}_x^{3/2}\times\dot{H}_x^{1/2})$ and $u\in L_{t,x}^{6}(K\times\mathbb{R}^5)$ which satisfy the {\it Duhamel formula},
\begin{align}
u(t)=S(t)(u_0,u_1)+\int_0^t \frac{\sin((t-t')|\nabla|)}{|\nabla|}|u(t')|^2u(t')dt',\label{eq_duhamel}
\end{align}
where $S(t)(u_0,u_1)=\cos(t|\nabla|)u_0+\frac{\sin(t|\nabla|)}{|\nabla|}u_1$ is the solution to the linear wave equation with initial data $(u_0,u_1)$.  

We refer to $I$ as the {\it interval of existence} of $u$, and we say that $I$ is the {\it maximal interval of existence} if $u$ cannot be extended to any larger time interval.  We say that $u$ is a {\it global solution} if $I=\mathbb{R}$, and that $u$ is a {\it blow-up solution} if $\lVert u\rVert_{L_{t,x}^{6}(I\times\mathbb{R}^d)}=\infty$.

Our main result in this paper is the following theorem:
\begin{theorem}
\label{thm1}
Suppose that $u:I\times\mathbb{R}^5\rightarrow\mathbb{R}$ is a solution to (NLW) with maximal interval of existence $I\subset\mathbb{R}$ satisfying the a priori bound
\begin{align*}
(u,u_t)\in L_t^\infty(I;\dot{H}_x^{3/2}\times\dot{H}_x^{1/2}).
\end{align*}
Then $I=\mathbb{R}$ and
\begin{align*}
\lVert u\rVert_{L_{t,x}^6(\mathbb{R}\times\mathbb{R}^5)}\leq C
\end{align*}
for some constant $C=C(\lVert (u,u_t)\rVert_{L_t^\infty(\dot{H}_x^{3/2}\times\dot{H}_x^{1/2})})$.  Furthermore, $u$ scatters both forward and backward in time, i.e. there exist $(u_0^{\pm},u_1^{\pm})\in\dot{H}_x^{s_c}\times\dot{H}_x^{s_c-1}$ such that
\begin{align*}
\left\lVert (u(t)-S(t)(u_0^\pm,u_1^\pm),u_t(t)-\partial_t S(t)(u_0^\pm,u_1^\pm)) \right\rVert_{\dot{H}_x^{s_c}\times\dot{H}_x^{s_c-1}}\rightarrow 0
\end{align*}
as $t\rightarrow\pm\infty$.
\end{theorem}

In the energy-critical case $s_c=1$ with the defocusing nonlinearity $|u|^{4/(d-2)}u$, global well-posedness and scattering for the nonlinear wave equation was established in a number of works: see, for instance, \cite{BahouriGerard,Grillakis1,Grillakis2,Kapitanskii,Nakanishi,Pecher,Rauch,ShatahStruwe1,ShatahStruwe,Struwe,Tao1}.  In particular, Struwe \cite{Struwe} established global well-posedness for radial inital data in dimension $d=3$, and Grillakis \cite{Grillakis1} removed this radial assumption.  Grillakis then established global well-posedness and persistence of regularity in dimensions $3\leq d\leq 5$ \cite{Grillakis2}, while this result was obtained for all dimensions $d\geq 3$ by Shatah and Struwe \cite{ShatahStruwe1,ShatahStruwe2,ShatahStruwe} and Kapitanskii \cite{Kapitanskii}.

In all of these works in the energy-critical case, the key properties used are an immediate uniform-in-time control of the critical norm $\dot{H}_x^{1}\times L_x^2$ by virtue of the conservation of energy, along with the uniform space-time control given by the Morawetz estimate,
\begin{align}
\int_I\int_{\mathbb{R}^d} \frac{|u(t,x)|^4}{|x|}dxdt\lesssim E(u(0)).\label{eq_morawetz}
\end{align}

Turning to the energy-supercritical regime $s_c>1$, the global well-posedness and scattering of solutions for large initial data remains an important open problem in this setting.  In particular, the lack of any known conserved quantity at the critical regularity renders the problem significantly more difficult than the energy-critical case.  Nevertheless, a number of recent works have treated the energy-supercritical nonlinear Schr\"odinger and nonlinear wave equations under an assumption of a uniform-in-time control of the critical norm, that is, 
\begin{align}
(u,u_t)\in L_t^\infty(I;\dot{H}_x^{s_c}\times\dot{H}_x^{s_c-1}),\label{apriori}
\end{align}
where $I\subset \mathbb{R}$ is the maximal interval of existence, in the spirit of the recent work of Escauriaza, Seregin and Sverak for Navier-Stokes \cite{ESS}.  

The first such result was obtained by Kenig and Merle in \cite{KMESupercriticalNLW} using the concentration compactness approach introduced in their study of the focusing energy-critical NLS and NLW \cite{KMECriticalNLS,KMECriticalNLW}.  In \cite{KMESupercriticalNLW}, global well posedness and scattering was obtained for solutions satisfying ($\ref{apriori}$) with radial initial data in dimension $d=3$ with the nonlinearity $|u|^pu$  with $p>4$, where the critical regularity becomes $s_c=\frac{d}{2}-\frac{2}{p}$; see also \cite{KModd}.  Killip and Visan removed the radial assumption in dimension $d=3$ \cite{KVNLW3} and established the result for radial initial data in dimensions $d\geq 3$ for a particular range of $p$ \cite{KVNLWradial}, which in dimension $d=5$ becomes $\frac{4}{3}<p<2$, disjoint from the cubic case treated in \cite{BulutRadial}.

In \cite{BulutCubic}, we initiated the study of global well-posedness and scattering under the assumed a priori bound in the case of the cubic nonlinearity in dimensions $d\geq 6$ with no radial assumption on the initial data, and in \cite{BulutRadial} we extended this result to cover dimension $d=5$ with radial initial data.  The role of the present work is therefore to complete the study of the defocusing cubic energy-supercritical NLW under the assumed uniform a priori control of the critical norm ($\ref{apriori}$).

\subsection{Outline of the proof of Theorem $\ref{thm1}$}

Our proof of Theorem $\ref{thm1}$ makes use of the concentration-compactness approach of Kenig and Merle, and proceeds as follows: 

Assuming that the theorem fails, the first step in the argument is to extract a minimal counterexample to the failure of global well-posedness and scattering; this counterexample will be referred to as a {\it minimal blow-up solution}.  More precisely, we recall the following result:
\begin{theorem}
[Construction of a minimal blow-up solution, \cite{KMECriticalNLW,KMESupercriticalNLW}]
\label{thm_minred}
Suppose that Theorem $\ref{thm1}$ failed.  Then there exists a solution $u:I\times\mathbb{R}^5\rightarrow\mathbb{R}$ to (NLW) with maximal interval of existence $I$,
\begin{align*}
(u,u_t)\in L_t^\infty(I;\dot{H}_x^{3/2}\times\dot{H}_x^{1/2}),\quad\textrm{and}\quad \lVert u\rVert_{L_{t,x}^{6}(I\times\mathbb{R}^5)}=\infty
\end{align*}
such that $u$ is a minimal blow-up solution in the following sense: for any solution $v$ with maximal interval of existence $J$ such that $\lVert v\rVert_{L_{t,x}^{6}(J\times\mathbb{R}^5)}=\infty$, we have
\begin{align*}
\sup_{t\in I} \lVert (u(t),u_t(t))\rVert_{\dot{H}_x^{3/2}\times\dot{H}_x^{1/2}}&\leq \sup_{t\in J} \lVert (v(t),v_t(t))\rVert_{\dot{H}_x^{3/2}\times\dot{H}_x^{1/2}}.
\end{align*}
Moreover, there exist $N:I\rightarrow \mathbb{R}^+$ and $x:I\rightarrow\mathbb{R}^5$ such that the set
\begin{align}
K&=\{(\frac{1}{N(t)}u(t,x(t)+\frac{x}{N(t)}),\,\frac{1}{N(t)^{2}}u_t(t,x(t)+\frac{x}{N(t)})):t\in I\},\label{cptness}
\end{align}
has compact closure in $\dot{H}^{s_c}\times\dot{H}^{s_c-1}(\mathbb{R}^5)$.
\end{theorem}
The proof of Theorem $\ref{thm_minred}$ is based on concentration compactness ideas in the form of a profile decomposition result for the linear wave equation, established in dimension $d=3$ by Bahouri and Gerard in \cite{BahouriGerard} and treated in dimensions $d\geq 3$ by the author in \cite{BulutMax}.  This profile decomposition states that, up to subsequences and the symmetries of the equation, any bounded sequence of initial data in the critical space $\dot{H}_x^{3/2}\times \dot{H}_x^{1/2}$ can be decomposed as a linear combination of pairwise orthogonal profiles and an error term which is small in a Strichartz norm.  More precisely, we have
\begin{proposition} [Profile decomposition \cite{BahouriGerard,BulutMax}]
\label{lab14}Let $(u_{0,n},u_{1,n})_{n\in\mathbb{N}}$ be a bounded sequence in $\dot{H}_x^{3/2}(\mathbb{R}^5)\times \dot{H}_x^{1/2}(\mathbb{R}^5)$.

Then there exists a subsequence of $(u_{0,n},u_{1,n})$ (still denoted $(u_{0,n},u_{1,n})$), a sequence of profiles $(V_0^j,V_1^j)_{j\in \mathbb{N}}\subset \dot{H}_x^{3/2}\times\dot{H}_x^{1/2}$, and a sequence of triples $(\epsilon_n^j,x_n^j,t_n^j)\in \mathbb{R}^+\times\mathbb{R}^5\times\mathbb{R}$, which are orthogonal in the sense that for every $j\neq j'$,
\begin{align*}
\frac{\epsilon_n^j}{\epsilon_n^{j'}}+\frac{\epsilon_n^{j'}}{\epsilon_n^j}+\frac{|t_n^j-t_n^{j'}|}{\epsilon_n^j}+\frac{|x_n^j-x_n^{j'}|}{\epsilon_n^j}\mathop{\longrightarrow}_{n\rightarrow\infty} \infty,
\end{align*}
and for every $\ell\geq 1$, if 
\begin{align*}
V^j=S(t)(V_0^j,V_1^j)\quad  and \quad V_n^j(t,x)=\frac{1}{(\epsilon_n^j)}V^j\left(\frac{t-t_n^j}{\epsilon_n^j},\frac{x-x_n^j}{\epsilon_n^j}\right),
\end{align*}
then
\begin{align*}
(u_{0,n}(x),u_{1,n}(x))&=\sum_{j=1}^\ell (V_n^j(0,x),\partial_t V_n^j(0,x))+(w_{0,n}^\ell(x),w_{1,n}^\ell(x))
\end{align*}
with
\begin{align*}
\limsup_{n\rightarrow\infty}\lVert S(t)(w_{0,n}^\ell,w_{1,n}^\ell)\rVert_{L_t^qL_x^r}\mathop{\longrightarrow}_{\ell\rightarrow\infty}0
\end{align*}
for every $(q,r)$ an $\dot{H}_x^{3/2}$-wave admissible pair with $q,r\in (2,\infty)$. For all $\ell\geq 1$, we also have, 
\begin{align*}
&\lVert u_{0,n}\rVert_{\dot{H}_x^{3/2}}^2+\lVert u_{1,n}\rVert_{\dot{H}_x^{1/2}}^2\\
&\hspace{0.2in}=\sum_{j=1}^\ell \left[\lVert V^j_0\rVert_{\dot{H}_x^{3/2}}^2+\lVert V^j_1\rVert_{\dot{H}_x^{1/2}}^2\right]+\lVert w_{0,n}^\ell\rVert_{\dot{H}_x^{3/2}}^2+\lVert w_{1,n}^\ell\rVert_{\dot{H}_x^{1/2}}^2+o(1),\quad n\rightarrow\infty. 
\end{align*}
\end{proposition}

The next step in the argument consists of showing that such a minimal counterexample cannot occur.  As we will see below, this failure of existence arises as a consequence of the compactness property ($\ref{cptness}$).  Before proceeding further, we now recall an equivalent formulation of ($\ref{cptness}$) from \cite{KVNLW3,KVNLWradial} which will be an important tool in our analysis of blow-up solutions.
\begin{definition}[almost periodic modulo symmetries] A solution $u$ to (NLW) with time interval $I$ is said to be almost periodic modulo symmetries if $(u,u_t)\in L_t^\infty(I;\dot{H}_x^{s_c}\times \dot{H}_x^{s_c-1})$ and there exist functions $N:I\rightarrow\mathbb{R}^+$, $x:I\rightarrow\mathbb{R}^5$ and $C:\mathbb{R}^+\rightarrow\mathbb{R}^+$ such that for all $t\in I$ and $\eta>0$,
\begin{align*}
\int_{|x-x(t)|\geq C(\eta)/N(t)} ||\nabla|^{s_c}u(t,x)|^2+||\nabla|^{s_c-1}u_t(t,x)|^2dx\leq \eta
\end{align*}
and
\begin{align*}
\int_{|\xi|\geq C(\eta)N(t)} ||\xi|^{3/2}\hat{u}(t,\xi)|^2+||\xi|^{1/2}\hat{u}_t(t,\xi)|^2d\xi\leq \eta.
\end{align*}
\end{definition}

Almost periodic solutions satisfy a reduced form of the Duhamel formula which states that the linear components of ($\ref{eq_duhamel}$) vanish as the time $T$ approaches the endpoints of the maximal interval of existence $I$.
\begin{proposition}[Duhamel formula for almost periodic solutions, \cite{TaoVisanZhang,KVNLW3}] \label{prop_duhamel} Let $u:I\times\mathbb{R}^5\rightarrow\mathbb{R}$ be a solution to (NLW) with maximal interval of existence $I$ which is almost periodic modulo symmetries.  Then for all $t\in I$,
\begin{align*}
&\bigg(\int_t^T\frac{\sin((t-t')|\nabla|)}{|\nabla|}F(u(t'))dt',\int_t^T \cos((t-t')|\nabla|)F(u(t'))dt'\bigg)\\
&\hspace{2in}\mathop{\rightharpoonup}_{T\rightarrow\sup I} (u(t),u_t(t)),
\end{align*}
and
\begin{align*}
&\bigg(-\int_T^t\frac{\sin((t-t')|\nabla|)}{|\nabla|}F(u(t'))dt',-\int_T^t \cos((t-t')|\nabla|)F(u(t'))dt'\bigg)\\
&\hspace{2in}\mathop{\rightharpoonup}_{T\rightarrow\inf I} (u(t),u_t(t)),
\end{align*}
weakly in $\dot{H}_x^{3/2}\times \dot{H}_x^{1/2}$.
\end{proposition}

We now recall a refinement of Theorem $\ref{thm_minred}$, which states that any failure of Theorem $\ref{thm1}$ not only implies the existence of a minimal blow-up solution, but also implies that there exists an almost periodic solution which belongs to one of three particular blow-up scenarios, for which the function $N(t)$ is specified further.  
\begin{theorem}[Identification of blow-up scenarios, \cite{KVNLW3}]
\label{thm_apred}
Suppose that Theorem $\ref{thm1}$ fails.  Then there exists a solution $u:I\times\mathbb{R}^5\rightarrow\mathbb{R}$ to (NLW) with maximal interval of existence $I$ such that $u$ is almost periodic modulo symmetries,
\begin{align*}
(u,u_t)\in L_t^\infty(\dot{H}_x^{3/2}\times\dot{H}_x^{1/2}),\quad\textrm{and}\quad \lVert u\rVert_{L_{t,x}^6(I\times\mathbb{R}^5)}=\infty,
\end{align*}
and $u$ satisfies one of the following:
\begin{itemize}
\item (finite time blowup solution) either $\sup I<\infty$ or $\inf I>-\infty$.
\item (soliton-like solution) $I=\mathbb{R}$ and $N(t)=1$ for all $t\in \mathbb{R}$.
\item (low-to-high frequency cascade solution) $I=\mathbb{R}$ and
\begin{align*}
\inf_{t\in\mathbb{R}} N(t)\geq 1,\quad\textrm{and}\quad \limsup_{t\rightarrow\infty} N(t)=\infty.
\end{align*}
\end{itemize}
\end{theorem}

To complete the proof of Theorem $\ref{thm1}$, it suffices to show that each of the scenarios identified in Theorem $\ref{thm_apred}$ cannot occur.

We begin with the finite time blow-up solution, which is ruled out in Proposition $\ref{prop_fintime}$.  The argument is based on the use of the finite speed of propagation and almost periodicity to show that up to a translation in space, the solution at a given time $t$ has compact support which shrinks to zero as $t$ approaches the blow-up time.  The conservation of energy is then used to contradict the property that the solution is a blow-up solution.  This argument is essentially the same as in higher dimensions (see \cite{BulutCubic} and the references cited therein); for completeness, we include a brief account of the proof in Section $6$.

We now turn to the remaining two scenarios, the soliton-like and low-to-high frequency cascade solutions.  As in higher dimensions, the main tool that we use to rule out these scenarios is to show that due to their particular properties, the minimal blow-up solutions in these classes have finite energy.  We then exploit the control given by the conseration of energy and the Morawetz estimate ($\ref{eq_morawetz}$) to show that these solutions cannot exist.

Before proceeding with the current case of dimension five, let us first recall the main steps used to obtain the finiteness of energy for these two scenarios in dimensions six and higher:
\begin{itemize}
\item We first refine the bound $u\in L_t^\infty L_x^d$ (which is immediate from the Sobolev embedding and the a priori bound $u\in L_t^\infty \dot{H}_x^{s_c}$) to $L_t^\infty L_x^p$ for some $p<d$.  In particular, we use a bootstrap argument to bound the low frequencies of $u$ via almost periodicity and the analogue of Proposition $\ref{prop_duhamel}$, while the high frequencies are bounded by the a priori bound.
\item We next use this $L_t^\infty L_x^p$ bound to establish an implication of the form
\begin{align}
&\hspace{0.4in}(u,u_t)\in L_t^\infty(\dot{H}_x^{s}\times\dot{H}_x^{s-1})\quad\leadsto \quad (u,u_t)\in L_t^\infty(\dot{H}_x^{s-s_0}\times\dot{H}_x^{s-1-s_0}).\label{eq_impl}
\end{align}
for some $s_0>0$.  This is accomplished by using the double Duhamel technique introduced in \cite{CKSTT} and used in the context of the concentration compactness method in \cite{KVECriticalNLS,KVESupercriticalNLS}.  More precisely, we consider the homogeneous Sobolev norm as an inner product and using the Duhamel formula provided by Proposition $\ref{prop_duhamel}$ forward and backward in time on each side of the inner product, respectively.  We remark that the resulting integrals are convergent in dimensions $d\geq 6$; this is the source of the dimensional restriction in our work \cite{BulutCubic}.
\item Once we obtain the second step above, we iterate this argument starting with the a priori bound $(u,u_t)\in L_t^\infty(\dot{H}_x^{s_c}\times\dot{H}_x^{s_c-1})$ to obtain the desired decay $L_t^\infty(\dot{H}_x^{1-\epsilon}\times\dot{H}_x^{-\epsilon})$ for some $\epsilon>0$.  In particular, we obtain that the energy is finite.
\end{itemize}

In the present work, in dimension $d=5$, we overcome the restriction in higher dimensions by using a covering argument to introduce a spatial localization into the Duhamel integrals.  This is inspired by a similar approach used in \cite{KVNLW3} to study the energy-supercritical NLW in dimension $d=3$.  However, substantial differences between the linear propagator, Strichartz and dispersive estimates in dimensions $d=3$ and $d=5$ require us to argue differently than in \cite{KVNLW3}.

We now outline the main tools involved in the present work as follows:
\begin{itemize}
\item A refined version of the improved $L_t^\infty L_x^p$ integrability obtained in dimensions $d\geq 6$.  In particular, we establish the bound $u\in L_t^\infty L_x^p$ for some range of $p<5$; this shows that the a priori bound $(u,u_t)\in L_t^\infty(\dot{H}_x^{3/2}\times\dot{H}_x^{1/2})$ implies that the soliton-like and low-to-high frequency cascade solutions belong to $L_t^\infty L_x^3$, and also enables us to exploit additional bounds of the form $(u,u_t)\in L_t^\infty(\dot{H}_x^{s}\times\dot{H}_x^{s-1})$ to conclude the integrability $u\in L_t^\infty L_x^p$ for an expanded range of $p$ (see Proposition $\ref{prop2}$).
\item A quantitative estimate on the decay of the $L_t^\infty L_x^5$ norm away from the centering function $x(t)$ (see Proposition $\ref{prop_quant}$).  Such a result was first observed in \cite{KVNLW3} with the corresponding norm $L_t^\infty L_x^{3p/2}$ in dimension $d=3$.  We remark that in our setting, Proposition $\ref{prop_quant}$ takes advantage of the improvement in $L_t^\infty L_x^p$ integrability given by Proposition $\ref{prop2}$ to give better decay when additional $L_t^\infty(\dot{H}_x^s\times\dot{H}_x^{s-1})$ bounds are present.  This improvement in decay is an essential aspect of our argument.
\item A {\it subluminality} property which expresses that for a suitable class of solutions, the centering function $x(t)$ travels strictly slower than the propagation speed for (NLW).  An analogous result was first obtained in \cite{KVNLW3} in the dimension $d=3$ setting.  We note that the main estimate used to obtain the subluminality in \cite{KVNLW3} is the energy-flux inequality, while this estimate no longer provides the decay required in the current $d=5$ setting.  To overcome this difficulty, we establish a frequency localized form of the energy-flux inequality which is accompanied by a frequency localized form of the concentration of potential energy (see Lemma $\ref{lem2new}$ and Lemma $\ref{lem3}$).
\end{itemize}

With these tools in hand, we establish an iterative improvement in decay properties for the solution $u$.  In particular, to make the Duhamel integrals convergent, we use a covering argument as in \cite{KVNLW3} to localize the resulting integrals in space, combined with a variant of a weak diffraction lemma from \cite{KVNLW3} adapted to the $d=5$ setting.

It is important to note that the applications of Proposition $\ref{prop2}$ and Proposition $\ref{prop_quant}$ in the iterative improvement of the decay in analogy to ($\ref{eq_impl}$) use only the a priori bound $(u,u_t)\in L_t^\infty(\dot{H}_x^{3/2}\times\dot{H}_x^{1/2})$ to obtain the decay $(u,u_t)\in L_t^\infty(\dot{H}_x^s\times\dot{H}_x^{s-1})$ for $s>\frac{5}{4}$, which is still above the energy level.  We then use this improved decay to obtain the integrability $L_t^\infty L_x^p$ for a larger range of $p$ than the one that is obtained via the a priori bound alone.  This extra integrability in turn allows us to continue the iteration argument and we conclude that the energy is finite.

We therefore again have access to the Morawetz estimate ($\ref{eq_morawetz}$) and the conservation of energy.  We then rule out the soliton-like and low-to-high frequency cascade scenarios, making use of almost periodicity, the finite speed of propagation, the Morawetz estimate ($\ref{eq_morawetz}$) and the conservation of energy.  This is accomplished in Proposition $\ref{prop_sol}$ and Proposition $\ref{prop_casc}$, respectively.

\subsection*{Outline of the paper}

In Section $2$, we establish our notation and recall some essential preliminaries concerning the wave equation, as well as some analytical tools which we will use in our arguments.  In section $3$, we establish improved integrability properties of the soliton-like and low-to-high frequency cascade solutions.  Then, in Section $4$ we establish a frequency-localized bound of energy-flux type and a frequency-localized form of the concentration of potential energy, which are then used to prove the subluminality result.  Finally, in Section $5$ we prove that the global scenarios have finite energy, and in Section $6$ we preclude each of the scenarios.

\section{Preliminaries}
\parskip=8pt

\subsection{Notation}
We will write $L_t^qL_x^r(I\times\mathbb{R}^5)$ to indicate the Banach space of functions $u:I\times\mathbb{R}^5\rightarrow\mathbb{R}$ with the space-time norm
\begin{align*}
\lVert u\rVert_{L_t^qL_x^r}:=\left(\int_{I} \lVert u(t)\rVert_{L_x^r(\mathbb{R}^5)}^qdt\right)^{1/q}<\infty
\end{align*}
with the standard convention when $q$ or $r$ is equal to infinity.  If $q=r$, we shorten the notation and write $L_{t,x}^q$ in place of $L_t^qL_x^q$.

We write $X\lesssim Y$ to mean that there exists a constant $C>0$ such that $X\leq CY$, while $X\lesssim_u Y$ indicates that the constant $C=C_u$ may depend on $u$.  We use the symbol $\nabla$ for the derivative operator in only the space variables.  

Throughout the exposition, we define the Fourier transform on $\mathbb{R}^5$ by
\begin{align*}
\widehat{f}(\xi)=(2\pi)^{-5/2}\int_{\mathbb{R}^5} e^{-ix\cdot \xi}f(x)dx.
\end{align*}
We also denote the homogeneous Sobolev spaces by $\dot{H}_x^s(\mathbb{R}^5)$, $s\in\mathbb{R}$, equipped with the norm
\begin{align*}
\lVert f\rVert_{\dot{H}_x^s}=\lVert |\nabla|^sf\rVert_{L_x^2}
\end{align*}
where the fractional differentiation operator is given by
\begin{align*}
\widehat{|\nabla|^sf}(\xi)&=|\xi|^s\widehat{f}(\xi).
\end{align*}

We now recall the explicit formulation of the free propagator for the wave equation as
\begin{align*}
S(t)(f,g)=\cos(t|\nabla|)f+\frac{\sin(t|\nabla|)}{|\nabla|}g,
\end{align*}
which is given in Fourier space by
\begin{align*}
\widehat{S(t)}(f,g)=\cos(t|\xi|)\hat{f}(\xi)+\frac{\sin(t|\xi|)}{|\xi|}\hat{g}(\xi).
\end{align*}
In dimension five, we have the additional representations
\begin{align}
\cos(t|\nabla|)f&=\frac{3}{|S^4|}\partial_t \bigg[t^{-1}\partial_t\bigg(t^{-1}\int_{\partial B(x,t)}f(y)dS(y)\bigg)\bigg]\label{lop1}
\end{align}
and
\begin{align}
\frac{\sin(t|\nabla|)}{|\nabla|}g&=\frac{3}{t|S^4|}\partial_t \left[\frac{1}{t}\int_{|y-x|=|t|} g(y)dS(y)\right].\label{lop2}
\end{align}
where $B(x,r)=\{x\in\mathbb{R}^5:|x|<r\}$ and $|S^4|$ denotes the surface area of the unit sphere $\partial B(0,1)\subset \mathbb{R}^5$.

We also recall the following expression of finite speed of propagation for almost periodic solutions, which we will frequently use in our arguments:
\begin{lemma}
[Speed of propagation for $x(t)$, \cite{KVNLW3,BulutCubic}]
\label{lem_a2}
Suppose that $u:\mathbb{R}\times\mathbb{R}^5\rightarrow \mathbb{R}$ is an almost periodic solution to (NLW) with $N(t)\geq 1$ for all $t\in\mathbb{R}$.
Then there exists $C>0$ such that for all $t,\tau\in\mathbb{R}$,
\begin{align}
\label{fsprop_xt}|x(t)-x(\tau)|\leq |t-\tau|+C(N(t)^{-1}+N(\tau)^{-1}).
\end{align}
\end{lemma}
For a proof of this lemma, we refer the reader to $(4.2)$ in \cite{KVNLW3} and, for a related argument in the case of the soliton-like solution \cite[Lemma $8.2$]{BulutCubic}.  Although the result in \cite{KVNLW3} is stated for dimension $d=3$, the proof applies equally to the current setting.

\subsection{Strichartz estimates for NLW}

For $s\geq 0$, we say that a pair of exponents $(q,r)$ is $\dot{H}_x^s$-{\it wave admissible} if $q,r\geq 2$, $r<\infty$ and it satisfies
\begin{align*}
\frac{1}{q}+\frac{2}{r}\leq 1,\quad \frac{1}{q}+\frac{5}{r}=\frac{5}{2}-s.
\end{align*}
We also define the following {\it Strichartz norms}.  For each $I\subset\mathbb{R}$ and $s\geq 0$, we set
\begin{align*}
\lVert u\rVert_{S_s(I)}&=\sup_{(q,r)\,\dot{H}_x^s-\textrm{wave admissible}} \lVert u\rVert_{L_t^qL_x^r(I\times\mathbb{R}^5)},\\
\lVert u\rVert_{N_s(I)}&=\inf_{(q,r)\,\dot{H}_x^s-\textrm{wave admissible}} \lVert u\rVert_{L_t^{q'}L_x^{r'}(I\times\mathbb{R}^5)}.
\end{align*}

Suppose $u:I\times\mathbb{R}^5\rightarrow\mathbb{R}$ with time interval $0\in I\subset\mathbb{R}$ is a solution to the nonlinear wave equation
\begin{align*}
\left\lbrace \begin{array}{rl}u_{tt}-\Delta u+F&=0\\
(u,u_t)|_{t=0}&=(u_0,u_1)\in \dot{H}_x^{\mu}\times \dot{H}_x^{\mu-1}(\mathbb{R}^5), \quad \mu\in\mathbb{R}.
\end{array}
\right.
\end{align*}
Then for all $s,\tilde{s}\in\mathbb{R}$ we have the {\it inhomogeneous Strichartz estimates} \cite{GinibreVelo,KeelTao},
\begin{align}
\nonumber &\lVert |\nabla|^s u\rVert_{S_{\mu-s}(I)}+\lVert |\nabla|^{s-1}u_t\rVert_{S_{\mu-s}(I)}\\
&\hspace{1.5in}\lesssim \lVert (u_0,u_1)\rVert_{\dot{H}_x^{\mu}\times\dot{H}_x^{\mu-1}}+\lVert |\nabla|^{\tilde{s}}F\rVert_{N_{1+\tilde{s}-\mu}(I)}.\label{str}
\end{align}

We also recall the standard dispersive estimate for the linear propagator (see, for instance \cite{ShatahStruwe}),
\begin{lemma}[Dispersive estimate] For any $2\leq p<\infty$ and $t\neq 0$, we have
\begin{align*}
\left\lVert \frac{e^{it|\nabla|}}{|\nabla|}f\right\rVert_{L_x^p}\lesssim |t|^{-2(1-\frac{2}{p})}\lVert |\nabla|^{2-\frac{6}{p}}f\rVert_{L_x^{p'}}.
\end{align*}
for all $f\in\mathcal{S}(\mathbb{R}^5)$, where $\frac{1}{p}+\frac{1}{p'}=1$.

In particular,
\begin{align}
\left\lVert \frac{\sin(t|\nabla|)}{|\nabla|}f\right\rVert_{L_x^p}\lesssim |t|^{-2(1-\frac{2}{p})}\lVert |\nabla|^{2-\frac{6}{p}}f\rVert_{L_x^{p'}},\label{eq_dispersive}
\end{align}
for all $f\in\mathcal{S}(\mathbb{R}^5)$.
\end{lemma}

\subsection{Basic Littlewood-Paley theory}

Let $\phi(\xi)$ be a real valued radially symmetric bump function supported in the ball $\{\xi\in\mathbb{R}^5:|\xi|\leq \frac{11}{10}\}$ which equals $1$ on the ball $\{\xi\in\mathbb{R}^5:|\xi|\leq 1\}$.  For any dyadic number $N=2^k$, $k\in\mathbb{Z}$, we define the following Littlewood-Paley operators:
\begin{align*}
\widehat{P_{\leq N}f}(\xi)&=\phi(\xi/N)\hat{f}(\xi),\\
\widehat{P_{>N}f}(\xi)&=(1-\phi(\xi/N)\hat{f}(\xi),\\
\widehat{P_Nf}(\xi)&=(\phi(\xi/N)-\phi(2\xi/N))\hat{f}(\xi).
\end{align*}

\noindent Similarly, we define $P_{<N}$ and $P_{\geq N}$ with 
\begin{align*}
P_{<N}=P_{\leq N}-P_N,\quad P_{\geq N}=P_{>N}+P_N,
\end{align*}
and also
\begin{align*}
P_{M<\cdot\leq N}:=P_{\leq N}-P_{\leq M}=\sum_{M<N_1\leq N} P_{N_1}
\end{align*}
whenever $M<N$.

These operators commute with one another and with derivative operators.  Moreover, they are bounded on $L_x^p$ for $1\leq p\leq \infty$ and obey the following {\it Bernstein inequalities}, 
\begin{align*}
\lVert |\nabla|^s P_{\leq N}f\rVert_{L_x^p}&\lesssim N^s\lVert P_{\leq N}f\rVert_{L_x^p},\\
\lVert P_{>N}f\rVert_{L_x^p}&\lesssim N^{-s}\lVert P_{>N}|\nabla|^sf\rVert_{L_x^p},\\
\lVert |\nabla|^{\pm s}P_Nf\rVert_{L_x^p}&\sim N^{\pm s} \lVert P_Nf\rVert_{L_x^p},
\end{align*}
with $s\geq 0$ and $1\leq p\leq\infty$.

In addition to the above Littlewood-Paley multiplier operators, we will also need the function $\theta:\mathbb{R}^5\rightarrow [0,\infty)$ defined by
\begin{align}
\theta(x)=\prod_{j=1}^5 \bigg(\frac{\sin(x_i)}{x_i}\bigg)^\alpha,\quad \alpha>1\label{eq_theta}
\end{align}
as well as the corresponding Fourier multiplier operator, given by
\begin{align*}
\widehat{\theta(\nabla)f}=\theta(\xi)\hat{f}(\xi).
\end{align*}
We observe that for all $x\in \mathbb{R}^5$, $0\leq \theta(x)\leq 1$, and that $\theta\in L^1$.

As mentioned in the introduction, the improvement of decay properties for $u$ in Section $\ref{sec_regularity}$ is based on a localization procedure.  Due to the nonlocal nature of the Littlewood-Paley and fractional derivative operators, as in \cite{KVNLW3} we will make use of the following lemma, which allows us to estimate the contributions far from the support of the localization.
\begin{lemma}[Mismatch estimates, \cite{KVNLW3}] \label{lem_mismatch} 
Let $\phi_1$ and $\phi_2$ be smooth functions on $\mathbb{R}^5$ such that $\max\{|\phi_1(x)|,|\phi_2(x)|\}\leq 1$ for all $x\in\mathbb{R}^5$ and such that there exists $A\geq 1$ with
\begin{align*}
\dist (\supp \phi_1,\supp \phi_2)&\geq A.
\end{align*}
Then for all $\sigma>0$ and $1\leq p\leq q\leq \infty$ we have
\begin{align*}
\lVert \phi_1|\nabla|^\sigma P_{\leq 1}(\phi_2 f)\rVert_{L_x^q(\mathbb{R}^5)}\lesssim A^{-\sigma-\frac{5}{p}+\frac{5}{q}}\lVert \phi_2 f\rVert_{L_x^p(\mathbb{R}^5)}.
\end{align*}
\end{lemma}

\section{Improved integrability properties}

In this section, we establish improved $L_t^\infty L_x^p$ integrability properties for a suitable class of almost periodic solutions to (NLW), which is the first step towards proving that the global solutions in the soliton-like and low-to-high frequency cascade scenarios have finite energy.  

In particular, in Proposition $\ref{prop2}$ we prove that these solutions belong to the space $L_t^\infty L_x^p$ for some range of $p<5$, having recalled that they belong to $L_t^\infty L_x^5$ due to the a priori bound $(u,u_t)\in L_t^\infty(\dot{H}_x^{3/2}\times\dot{H}_x^{1/2})$ combined with the Sobolev embedding.  On the other hand, in Proposition $\ref{prop_quant}$ we establish an improved integrability property in a different direction, by obtaining a decay estimate on the $L_t^\infty L_x^5$ norm.  

\subsection{Improved $L_t^\infty L_x^p$ integrability} We begin with Proposition $\ref{prop2}$, which provides $u\in L_t^\infty L_x^p$ integrability for some range of $p<5$.  As mentioned in the introduction, an important aspect of our statement is the refinement of this estimate to provide improved $L_t^\infty L_x^p$ decay when additional knowledge of the decay properties of $u$ in the form $(u,u_t)\in L_t^\infty(\dot{H}_x^{s}\times\dot{H}_x^{s-1})$, $1<s<\frac{3}{2}$, is present.
\begin{proposition}
\label{prop2}
Suppose that $u:\mathbb{R}\times\mathbb{R}^5\rightarrow\mathbb{R}$ is an almost periodic solution to (NLW) satisfying 
\begin{align*}
(u,u_t)\in L_t^\infty (\dot{H}_x^{s}\times\dot{H}_x^{s-1})\cap L_t^\infty(\dot{H}_x^{3/2}\times\dot{H}_x^{1/2}).
\end{align*}
for some $s\in (1,\frac{3}{2}]$.  Assume also that $\inf_{t\in \mathbb{R}} N(t)\geq 1$.

Then $u\in L_t^\infty L_x^p$ for every $p\in (\frac{10(3+2s)}{15+4s},5]$.
\end{proposition}

Before proceeding with the proof, we note that Proposition $\ref{prop2}$ in particular implies that for any global almost periodic solution $u$ to (NLW) satisfying the a priori bound
\begin{align*}
(u,u_t)\in L_t^\infty(\mathbb{R};\dot{H}_x^{3/2}\times\dot{H}_x^{1/2})\quad\textrm{with}\quad \inf_{t\in\mathbb{R}} N(t)\geq 1,
\end{align*}
we have 
\begin{align}
u\in L_t^\infty L_x^3.\label{eqrem2}
\end{align}

\begin{proof}[Proof of Proposition $\ref{prop2}$.]
Our argument follows the broad outline established in the proof of Proposition $5.1$ in \cite{KVNLW3}; see also Lemma $7.2$ in \cite{BulutCubic}.  Fix $\eta>0$ to be determined later in the argument, and let $u$ be a solution to (NLW) as stated in Proposition $\ref{prop2}$.  Since $u$ is an almost periodic solution with $N(t)\geq 1$ for all $t\in\mathbb{R}$, we may find a dyadic number $N_0\in (0,1)$ such that
\begin{align*}
\lVert |\nabla|^{3/2}u_{\leq N_0}\rVert_{L_t^\infty L_x^2}<\eta.
\end{align*}

As in \cite{KVNLW3}, our proof of Proposition $\ref{prop2}$ will be based on a recurrence argument followed by an application of a discrete Gronwall inequality.  Towards this end, we fix $r>5$ and define the quantity
\begin{align*}
\mathcal{S}(N)=N^{(5-r)/r}\lVert u_N\rVert_{L_t^\infty L_x^r}
\end{align*}

We begin by noting that by Bernstein's inequalities followed by Sobolev's embedding and the a priori bound $u\in L_t^\infty \dot{H}_x^{3/2}$, we have
\begin{align*}
\mathcal{S}(N)\lesssim N^\frac{3}{2}\lVert u_N\rVert_{L_t^\infty L_x^2}\lesssim \lVert |\nabla|^{3/2}u_N\rVert_{L_t^\infty L_x^2}<\infty 
\end{align*}
for $N>0$.

The following lemma then gives the necessary recurrence formula for the quantity $\mathcal{S}(N)$.
\begin{lemma}
\label{lem1a}
For all dyadic numbers $N\leq 8N_0$, we have
\begin{align}
\nonumber \mathcal{S}(N)&\lesssim_u \left(\frac{N}{N_0}\right)^{\frac{3}{2}}+\eta^2\sum_{\frac{2N}{8}\leq N_1\leq N_0} \left(\frac{N}{N_1}\right)^{2}\mathcal{S}(N_1)\\
&\hspace{0.2in}+\eta^2\sum_{N_1\leq \frac{N}{8}} \left(\frac{N_1}{N}\right)^{1/2}\mathcal{S}(N_1).\label{lem1clm1}
\end{align}
In particular, for every $\epsilon>0$ we have 
\begin{align}
\mathcal{S}(N)\lesssim_{u,\epsilon} N^{\frac{3}{2}-\epsilon}\label{lem1clm2}
\end{align}
for all $N\leq 8N_0$.
\end{lemma}

\begin{proof}
Fix $N\leq 8N_0$.  We first note that by the time translation symmetry, it suffices to prove the claim when $t=0$.  Using the reduced form of the Duhamel formula given in Proposition $\ref{prop_duhamel}$ combined with Minkowski's inequality, we obtain
\begin{align}
\nonumber N^{\frac{5}{r}-1}\lVert u_N(0)\rVert_{L_x^r}&\leq N^{\frac{5}{r}-1}\left \lVert \int_0^{N^{-1}} \frac{\sin(-t'|\nabla|)}{|\nabla|}P_NF(u(t'))dt'\right\rVert_{L_x^r}\\
\nonumber &\hspace{0.2in}+N^{\frac{5}{r}-1}\left\lVert \int_{N^{-1}}^{\infty} \frac{\sin(-t'|\nabla|)}{|\nabla|}P_NF(u(t'))dt'\right\rVert_{L_x^r}\\
&=:(I)+(II).
\end{align}

To estimate $(I)$, we use the Bernstein inequality combined with the dispersive estimate ($\ref{eq_dispersive}$) to get 
\begin{align}
\nonumber (I)&\lesssim \int_0^{N^{-1}} N^{3/2}\left\lVert \frac{\sin(-t'|\nabla|)}{|\nabla|}P_NF(u(t'))\right\rVert_{L_x^2}dt'\\
\nonumber &\lesssim \int_0^{N^{-1}}N^{3/2} \lVert |\nabla|^{-1}P_NF(u(t'))\rVert_{L_x^2}dt'\\
&\lesssim N\lVert P_NF(u(t))\rVert_{L_t^\infty L_x^{5/4}}.\label{eqabc1}
\end{align}

Similarly,
\begin{align}
\nonumber (II)&\lesssim \int_{N^{-1}}^\infty \left\lVert \frac{\sin(-t'|\nabla|)}{|\nabla|}P_NF(u(t'))\right\rVert_{L_x^5}dt'\\
\nonumber &\lesssim \int_{N^{-1}}^\infty |t|^{-6/5}\lVert |\nabla|^{4/5}P_NF(u(t))\rVert_{L_t^\infty L_x^{5/4}}\\
&\lesssim N\lVert P_NF(u(t))\rVert_{L_t^\infty L_x^{5/4}}.\label{eqabc2}
\end{align}

Collecting ($\ref{eqabc1}$) and ($\ref{eqabc2}$), we have
\begin{align*}
N^{\frac{5}{r}-1}\lVert u_N(0)\rVert_{L_x^r}&\lesssim_u N\lVert P_NF(u(t))\rVert_{L_t^\infty L_x^{5/4}}.
\end{align*}

We next estimate the right hand side of this inequality.  Decomposing $u$ as
\begin{align*}
u=u_{\leq N/8}+u_{N/8<\cdot \leq N_0}+u_{>N_0}=:u_1+u_2+u_3,
\end{align*}
we write
\begin{align*}
\lVert P_NF(u)\rVert_{L_t^\infty L_x^{5/4}}&\lesssim \lVert P_N(u_1^3)\rVert_{L_t^\infty L_x^{5/4}}+\lVert P_N(u_2^3)\rVert_{L_t^\infty L_x^{5/4}}\\
&\hspace{0.2in}+\sum_{i,j=1}^3 \lVert P_N(u_3u_iu_j)\rVert_{L_t^\infty L_x^{5/4}}+\sum_{i=1}^2 \lVert P_N(u_2u_1u_i)\rVert_{L_t^\infty L_x^{5/4}}\\
&=:(I)+(II)+\sum_{i,j=1}^3 (III)_{i,j}+\sum_{i=1}^2 (IV)_i.
\end{align*}

To establish ($\ref{lem1clm1}$, it now remains to estimate each of the terms $(I)$, $(II)$, $(III)_{i,j}$ and $(IV)_i$ individually.  

We first remark that $(I)=0$, since the support of the Fourier transform of $u_{\leq N/8}(t)^3$ gives that $P_N[u_{\leq N/8}(t)^3]\equiv 0$.  

To estimate the term $(II)$, we use H\"older's inequality along with the Sobolev embedding and Bernstein's inequality to obtain
\begin{align}
\nonumber (II)&\lesssim N\sum_{2N/8\leq N_1\leq N_2\leq N_3\leq N_0} \lVert u_{N_1}u_{N_2}u_{N_3}\rVert_{L_t^\infty L_x^{1}}\\
\nonumber &\leq N\sum_{2N/8\leq N_1\leq N_2\leq N_3\leq N_0} \lVert u_{N_1}\rVert_{L_t^\infty L_x^r}\lVert u_{N_2}\rVert_{L_t^\infty L_x^\frac{2r}{r-1}}\lVert u_{N_3}\rVert_{L_t^\infty L_x^\frac{2r}{r-1}}\\
\nonumber &\lesssim N\sum_{2N/8\leq N_1\leq N_2\leq N_3\leq N_0} N_1^{(r-5)/r}\mathcal{S}(N_1)(N_2N_3)^{\frac{5}{2r}-\frac{3}{2}}\lVert u_{\leq N_0}\rVert_{L_t^\infty \dot{H}_x^{3/2}}^2\\
\nonumber &\lesssim_u \eta^2N\sum_{2N/8\leq N_1\leq N_0} \mathcal{S}(N_1)N_1^{-2}\\
&\lesssim_u \eta^2N^{-1}\sum_{2N/8\leq N_1\leq N_0} \left(\frac{N}{N_1}\right)^{2}\mathcal{S}(N_1).\label{eqabcde1}
\end{align}

Turning to the terms $(III)_{i,j}$, $i,j=1,2,3$, we note that for each such $i$ and $j$, H\"older's inequality followed by the Bernstein and Sobolev inequalities gives
\begin{align}
\nonumber (III)_{i,j}&\lesssim N^{1/2}\lVert u_{>N_0}u_iu_j\rVert_{L_t^\infty L_x^{10/9}}\\
\nonumber &\leq N^{1/2}\lVert u_{>N_0}\rVert_{L_t^\infty L_x^2}\lVert u_i\rVert_{L_t^\infty L_x^5}\lVert u_j\rVert_{L_t^\infty L_x^5}\\
\nonumber &\lesssim_u N^{1/2}N_0^{-3/2}\lVert u\rVert_{L_t^\infty \dot{H}_x^{3/2}}^3\\
&\lesssim_u N^{1/2}N_0^{-3/2}\label{eqabcde2}
\end{align}

We now estimate the terms $(IV)_{i}$, $i=1,2$, for which we examine each case separately.  For the term $(IV)_1$, we write
\begin{align}
\nonumber (IV)_1&\lesssim \lVert u_{\frac{N}{8}<\cdot \leq N_0}u_{\leq \frac{N}{8}}^2\rVert_{L_t^\infty L_x^{5/4}}\\
\nonumber &\leq \lVert u_{\frac{N}{8}<\cdot \leq N_0}\rVert_{L_t^\infty L_x^2}\sum_{N_1\leq N_2\leq \frac{N}{8}}\lVert u_{N_1}\rVert_{L_{t,x}^\infty}\lVert u_{N_2}\rVert_{L_t^\infty L_x^{10/3}}\\
\nonumber &\lesssim_u N^{-3/2}\eta\sum_{N_1\leq N_2\leq \frac{N}{8}} N_1^{\frac{5}{r}}\lVert u_{N_1}\rVert_{L_t^\infty L_x^r}\lVert \nabla u_{N_2}\rVert_{L_t^\infty L_x^2}\\
\nonumber &\lesssim_u \eta^2N^{-3/2}\sum_{N_1\leq N_2\leq \frac{N}{8}} N_1\mathcal{S}(N_1)N_2^{-1/2}\\
&\lesssim_u N^{-1}\eta^2\sum_{N_1\leq \frac{N}{8}} \left(\frac{N_1}{N}\right)^{1/2}\mathcal{S}(N_1).\label{eqabcde3}
\end{align}

On the other hand, to estimate $(IV)_2$, we write
\begin{align}
\nonumber (IV)_2&\lesssim \lVert u_{\frac{N}{8}<\cdot \leq N_0}^2u_{\leq\frac{N}{8}}\rVert_{L_t^\infty L_x^{5/4}}\\
\nonumber &\leq \lVert u_{\frac{N}{8}<\cdot \leq N_0}\rVert_{L_t^\infty L_x^2}\lVert u_{\frac{N}{8}<\cdot \leq N_0}\rVert_{L_t^\infty L_x^{10/3}}\sum_{N_1\leq \frac{N}{8}}\lVert u_{N_1}\rVert_{L_{t,x}^\infty}\\
\nonumber &\lesssim_u \eta N^{-3/2}\lVert \nabla u_{\frac{N}{8}<\cdot\leq N_0}\rVert_{L_t^\infty L_x^2}\sum_{N_1\leq \frac{N}{8}} N_1^{\frac{5}{r}}\lVert u_{N_1}\rVert_{L_t^\infty L_x^r}\\
&\lesssim_u \eta^2 N^{-1}\sum_{N_1\leq \frac{N}{8}}\left(\frac{N_1}{N}\right)\mathcal{S}(N_1).\label{eqabcde4}
\end{align}
Putting the estimates $(\ref{eqabcde1})$-$(\ref{eqabcde4})$ together, and observing that $N_0\leq 1$, we obtain ($\ref{lem1clm1}$) as desired.

To obtain $(\ref{lem1clm2})$, we first recall the following discrete version of the Gronwall inequality from \cite{KVNLW3}:
\begin{lemma}
Let $\gamma, \gamma', C, \eta>0$ and $\rho\in (0,\gamma)$ be given such that
\begin{align*}
\eta\leq \frac{1}{4}\min\{1-2^{-\gamma},1-2^{-\gamma'},1-2^{\rho-\gamma}\}.
\end{align*}
Then for every bounded sequence $\{x_k\}\subset\mathbb{R}^+$ satisfying
\begin{align*}
x_k\leq C2^{-\gamma k}+\eta \sum_{l=0}^{k-1} 2^{-\gamma(k-l)}x_l+\eta\sum_{l=k}^\infty 2^{-\gamma' |k-l|}x_l,
\end{align*}
we have $x_k\leq (4C+\lVert x\rVert_{\ell^\infty})2^{-\rho k}$.
\end{lemma}
Fixing $0<\epsilon<\frac{3}{2}$ and using this lemma with $\gamma=\frac{3}{2}$, $\gamma'=\frac{1}{2}$, $\rho<\gamma$ and $\eta>0$ sufficiently small, we obtain 
\begin{align*}
\mathcal{S}(N)\lesssim_{u,\epsilon} N^{\frac{3}{2}-\epsilon}
\end{align*}
as desired (for a detailed description of a similar type of estimate see \cite{KVNLW3} and Appendix A of \cite{BulutCubic}).  This completes the proof of Lemma $\ref{lem1a}$. 
\end{proof}

Having obtained the recurrence formula in Lemma $\ref{lem1a}$, we are now ready to complete the proof of Proposition $\ref{prop2}$.
We begin by recalling that the bound $u\in L_t^\infty L_x^5$ follows from the a priori bound $(u,u_t)\in L_t^\infty(\dot{H}_x^{3/2}\times\dot{H}_x^{1/2})$ by the Sobolev embedding.  

Let $s\in (1,\frac{3}{2}]$ and $p\in (\frac{10(3+2s)}{15+4s},\frac{4s+10}{5})$ be given.  
Then, using the Littlewood-Paley decomposition, Bernstein inequalities, and interpolation, we obtain
\begin{align}
\nonumber \lVert u\rVert_{L_t^\infty L_x^p}
&\lesssim \sum_{N<8N_0} \lVert u_N\rVert_{L_t^\infty L_x^p}+\sum_{N\geq 8N_0} N^{1-\frac{5}{p}}\lVert |\nabla|^{3/2}u_N\rVert_{L_t^\infty L_x^2}\\
\nonumber &\lesssim_u \sum_{N<8N_0} \lVert u_N\rVert_{L_t^\infty L_x^2}^\frac{2(r-p)}{p(r-2)}\lVert u_N\rVert_{L_t^\infty L_x^r}^\frac{r(p-2)}{p(r-2)}+\sum_{N\geq 8N_0} N^{1-\frac{5}{p}}\\
\nonumber &\lesssim_u \sum_{N<8N_0} (N^{-s}\lVert |\nabla|^{s}u_N\rVert_{L_t^\infty L_x^2})^\frac{2(r-p)}{p(r-2)}\lVert u_N\rVert_{L_t^\infty L_x^r}^\frac{r(p-2)}{p(r-2)}+\sum_{N\geq 8N_0} N^{1-\frac{5}{p}}\\
\nonumber &\lesssim_u 1+\sum_{N<8N_0} N^{-\frac{2s(r-p)}{p(r-2)}}(N^{1-\frac{5}{r}}\mathcal{S}(N))^\frac{r(p-2)}{p(r-2)}\\
&\lesssim_u 1+\sum_{N<8N_0} N^{-\frac{2s(r-p)}{p(r-2)}}(N^{1-\frac{5}{r}+\left(\frac{3}{2}-\epsilon\right)})^\frac{r(p-2)}{p(r-2)}.\label{eqasdf}
\end{align}
for $r>5$ and $\epsilon>0$.  Choosing $r=\frac{20s+4ps-35p+70}{2(10+4s-5p)}$, $\epsilon=\frac{(10+4s-5p)(4sp+15p-30-20s)}{4(p-2)(20s+4ps-35p+70)}$, we obtain 
\begin{align*}
\lVert u\rVert_{L_t^\infty L_x^p}&\lesssim_u 1+\sum_{N<8N_0} N^\frac{(10+4s-5p)(4sp+15p-30-20s)}{4p(4s-15p+30+4sp)}<\infty,
\end{align*}
where we observed that $p\in (\frac{10(3+2s)}{15+4s},\frac{4s+10}{5})$ gives $\frac{(10+4s-5p)(4sp+15p-30-20s)}{4p(4s-15p+30+4sp)}>0$.  For the remaining values of $p$, the claim now follows by interpolating this bound with the bound $u\in L_t^\infty L_x^5$, which completes the proof of Proposition $\ref{prop2}$.
\end{proof}

\subsection{Quantitative $L_t^\infty L_x^5$ decay} We now establish a decay estimate for the $L_t^\infty L_x^5$ norm of global almost periodic solutions $u$ to (NLW) with $\inf N(t)\geq 1$.  

More precisely, we obtain the following proposition:
\begin{proposition} \label{prop_quant} Suppose $u:\mathbb{R}\times\mathbb{R}^5\rightarrow\mathbb{R}$ is an almost periodic solution to (NLW) with $\inf_{t\in\mathbb{R}} N(t)\geq 1$,
\begin{align*}
(u,u_t)\in L_t^\infty (\dot{H}_x^{3/2}\times\dot{H}_x^{1/2})\quad\textrm{and}\quad u\in L_t^\infty L_x^w
\end{align*}
for some $w\in (2,3]$.  Then for every $\epsilon>0$ there exists $C_u>0$ such that for every $R\geq 1$,
\begin{align*}
\sup_{t\in\mathbb{R}}\int_{|x-x(t)|\geq R} |u(t,x)|^5dx\leq C_uR^{-\frac{8(5-w)}{12-w}}.
\end{align*}
\end{proposition}

Our proof of Proposition $\ref{prop_quant}$ is based on rewriting the solution $u(t,x)$ using the Duhamel formula given by Proposition $\ref{prop_duhamel}$.  Accordingly, we will need the following estimate for the long-time portion of the Duhamel formula.
\begin{lemma}
\label{lemfive}
Suppose that $u:\mathbb{R}\times\mathbb{R}^5\rightarrow\mathbb{R}$ is an almost periodic solution to (NLW) with $(u,u_t)\in L_t^\infty (\dot{H}_x^{3/2}\times\dot{H}_x^{1/2})$ and $\inf_{t\in\mathbb{R}} N(t)\geq 1$.  Fix $4<p\leq 12$.  Then there exists $C_u>0$ such that for every $0<S<T$ we have
\begin{align*}
\left\lVert \int_S^T \frac{\sin(-t|\nabla|)}{|\nabla|}F(u(t))dt\right\rVert_{L_x^p}&\leq C_uS^{\frac{4}{p}-1}.
\end{align*}
\end{lemma}
\begin{proof}
Let $0<S<T$ be given.  By an application of Minkowski's inequality followed by the dispersive estimate ($\ref{eq_dispersive}$), we write
\begin{align*}
\left\lVert \int_S^T \frac{\sin(-t|\nabla|)}{|\nabla|}F(u(t))dt\right\rVert_{L_x^p}&\lesssim \int_S^T\left\lVert \frac{\sin(-t|\nabla|)}{\nabla|}F(u(t))\right\rVert_{L_x^p} dt\\
&\lesssim \int_S^T {t}^{-2(1-\frac{2}{p})} \lVert |\nabla|^{2-\frac{6}{p}}F(u)\rVert_{L_t^\infty L_x^{p/(p-1)}}dt\\
&\lesssim \int_S^T (t)^{\frac{4}{p}-2} dt\\
&\lesssim S^{\frac{4}{p}-1}
\end{align*}
where we have used interpolation to obtain 
\begin{align*}
\lVert |\nabla|^{2-\frac{6}{p}}F(u)\rVert_{L_t^\infty L_x^{p/(p-1)}}&\lesssim \lVert u\rVert^2_{L_t^\infty L_x^{\frac{10p}{2p+1}}} \lVert |\nabla|^{2-\frac{6}{p}}u \rVert_{L_t^\infty L_x^{\frac{5p}{3(p-2)}}}\\
&\lesssim \lVert u \rVert_{L_t^\infty L_x^3}^{\frac{3}{2p}} \lVert u \rVert_{L_t^\infty L_x^5}^{\frac{4p-3}{2p}}\lVert u\rVert_{L_t^\infty\dot{H}_x^{3/2}}\\
&\lesssim_u 1,
\end{align*}
where we have noted that $u\in L_t^\infty L_x^3$ as a consequence of Proposition $\ref{prop2}$.
\end{proof}

With this lemma in hand, we now are now ready to prove the quantitative decay result on the $L_t^\infty L_x^5$ norm, Proposition $\ref{prop_quant}$.
\begin{proof}[Proof of Proposition $\ref{prop_quant}$.]
We argue as in the proof of Proposition $6.1$ in \cite{KVNLW3}, and begin by noting that by the space and time translation symmetries, we may assume without loss of generality that $t=0$ and $x(0)=0$.   Let $\epsilon>0$ be given, fix $\delta>0$ to be determined, and define 
\begin{align*}
f(x):=\int_{\delta R}^\infty \frac{\sin(-t'|\nabla|)}{|\nabla|}F(u(t'))dt',\quad 
g(x):=\int_0^{\delta R} \frac{\sin(-t'|\nabla|)}{|\nabla|}F(u(t',x))dt'.
\end{align*}

Note that by using the no-waste Duhamel formula, followed by H\"older's inequality, Young's inequality and Lemma $\ref{lemfive}$ with $p=12$, we obtain the bound
\begin{align}
&\nonumber \int_{|x|\geq R} |u(0,x)|^5 dx\\
\nonumber &\hspace{0.2in}=\int_{|x|\geq R} |f(x)+g(x)|^{\frac{12(5-w)}{12-w}} |u(0,x)|^{\frac{7w}{12-w}} dx\\
\nonumber &\hspace{0.2in}\lesssim \int_{|x|\geq R} |f(x)|^{\frac{12(5-w)}{12-w}} |u(0,x)|^{\frac{7w}{12-w}} + |g(x)|^{\frac{12(5-w)}{12-w}}|u(0,x)|^{\frac{7w}{12-w}}dx\\
\nonumber &\hspace{0.2in}\lesssim_u \lVert f\rVert_{L_x^{12}}^\frac{12(5-w)}{12-w}\lVert u(0)\rVert_{L_x^{w}}^{\frac{7w}{12-w}}\\
\nonumber &\hspace{0.4in}+\epsilon^{-\frac{5(12-w)}{12(5-w)}}\int_{|x|\geq R} |g(x)|^5dx+\epsilon^{\frac{5(12-w)}{7w}}\int_{|x|\geq R} |u(0,x)|^5dx\\
\nonumber &\hspace{0.2in}\lesssim_u (\delta R)^{-\frac{8(5-w)}{12-w}}\\
\nonumber &\hspace{0.4in}+\epsilon^{-\frac{5(12-w)}{12(5-w)}}\int_{|x|\geq R} |g(x)|^5dx+\epsilon^{\frac{5(12-w)}{7w}}\int_{|x|\geq R} |u(0,x)|^5dx
\end{align}
for $\epsilon>0$.  Choosing $\epsilon>0$ sufficiently small (depending on the implicit constant), this yields
\begin{align}
\int_{|x|\geq R} |u(0,x)|^5dx&\lesssim_u (\delta R)^{-\frac{8(5-w)}{12-w}}+\lVert g\rVert_{L_x^5(|x|\geq R)}^5.\label{eqqa}
\end{align}

On the other hand, from the finite speed of propagation for the linear wave equation followed by the Sobolev embedding, we obtain
\begin{align}
\nonumber \lVert g\rVert_{L_x^5(|x|\geq R)}&\lesssim \left\lVert \int_0^{\delta R} \frac{\sin(-t'|\nabla|)}{|\nabla|}F(u(t'))dt'\right\rVert_{L_x^{5}(|x|\geq R)}\\
\nonumber &\lesssim \left\lVert \int_0^{\delta R} \frac{\sin(t'|\nabla|)}{|\nabla|}[F(u(t')\chi_{\Omega}(x))]dt'\right\rVert_{L_x^5}\\
&\lesssim \left\lVert \int_0^{\delta R} \frac{\sin(t'|\nabla|)}{|\nabla|}[F(u(t')\chi_{\Omega}(x))]dt'\right\rVert_{\dot{H}_x^{3/2}}\label{eq_gbound}
\end{align}
for any $\Omega\subset \mathbb{R}^5$ with $\{x:|x|\geq R-\delta R\}\subset \Omega$, where $\chi_\Omega$ is a smooth approximation to the characteristic function of $\Omega$.  

We now recall that since $x(0)=0$, the bound ($\ref{fsprop_xt}$) gives the existence of $C>0$ such that
\begin{align}
|x(t)|\leq |t|+C\leq \delta R+C\label{eqabcdef1}
\end{align}
for $t\in [0,\delta R]$.  We may therefore take $\Omega=\{x:|x-x(t)|\geq R-2\delta R-C\}$.  

Substituting ($\ref{eqabcdef1}$) into ($\ref{eq_gbound}$) and applying the inhomogeneous Strichartz estimate (backwards in time, applied to the solution of (NLW) with Cauchy data $(0,0)$ at $t_0=\delta R$), we obtain
\begin{align}
\nonumber (\ref{eq_gbound})&\lesssim \lVert |\nabla|^{5/4} F(u(t)\chi_{\Omega}(x))\rVert_{L_t^{2}([0,\delta R];L_x^{4/3})}\\
\nonumber &\lesssim \lVert u(t)\chi_{\Omega}(x)\rVert_{L_t^\infty([0,\delta R];L^{5}_x)}\lVert u(t)\chi_{\Omega}(x)\rVert_{L_t^{2}L_x^{10}}\lVert |\nabla|^{5/4} u(t)\chi_{\Omega}(x)\rVert_{L_t^\infty L_x^{20/9}}\\
\label{eq123123123}&\lesssim B(R/2)^{1/5}\lVert u(t)\chi_{\Omega}(x)\rVert_{L_t^{2}L_x^{10}}\lVert u(t)\chi_{\Omega}(x)\rVert_{L_t^\infty \dot{H}_x^{3/2}},
\end{align}
provided that $R$ is sufficiently large and $\delta$ sufficiently small, where
\begin{align*}
B(R):=\sup_{t\in\mathbb{R}}\int_{|x-x(t)|\geq R} |u(t,x)|^5dx.
\end{align*}

Fix $\eta>0$ to be determined.  Using the finite speed of propagation, we may now choose $R_0=R_0(\eta,u)>0$ such that if $\Omega\subset \{x:|x|\geq R_0\}$, then
\begin{align*}
\lVert u(t)\chi_{\Omega}\rVert_{L_t^2L_x^{10}}+\lVert u(t)\chi_{\Omega}\rVert_{L_t^\infty \dot{H}_x^{3/2}}<\eta,
\end{align*}
which in conjunction with ($\ref{eq123123123}$) gives
\begin{align}
\lVert g\rVert_{L_x^5(|x|\geq R)}&\lesssim \eta^2B(R/2)^{1/5}\label{eq_gbound1}
\end{align}
for $R\geq R_0$.

Combining ($\ref{eqqa}$) with ($\ref{eq_gbound1}$), we obtain
\begin{align*}
\int_{|x|\geq R}|u(0,x)|^5dx\lesssim_u (\delta R)^{-\frac{8(5-w)}{12-w}}+\eta^{10}B(R/2).
\end{align*}
Applying the space and time translation symmetries, 
\begin{align*}
B(R)\leq C'_u(\delta R)^{-\frac{8(5-w)}{12-w}}+\eta^{10}C'_uB(R/2)
\end{align*}
for each $\eta>0$ and every $R\geq R_0(\eta,u)$.

On the other hand, by the a priori bound $(u,u_t)\in L_t^\infty(\dot{H}_x^{3/2}\times\dot{H}_x^{1/2})$ we have $B(R)\lesssim 1$ for all $R<R_0(\eta,u)$.  Invoking an induction argument and choosing $\eta>0$ sufficiently small (to prevent the constants from blowing up in the induction), we obtain
\begin{align*}
B(R)\leq C_u R^{-\frac{8(5-w)}{12-w}}
\end{align*}
for all $R$ as desired.
\end{proof}
\section{Bounds of energy-flux type and subluminality}
\label{sec_eflux}

In this section, our main goal is to establish an improvement upon the classical finite speed of propagation for (NLW) known as {\it subluminality}, which states that for the soliton-like and low-to-high frequency cascade solutions, the centering function $x(t)$ in the definition of almost periodicity travels strictly slower than the propagation speed of the equation; this is the content of Theorem $\ref{thm_sublum}$.  The main tools we will use to establish this property are a frequency localized form of the energy-flux bound, which will be established in Lemma $\ref{lem2new}$, and a frequency localized version of the concentration of potential energy, Lemma $\ref{lem3}$.

We begin by recalling a form of the standard energy flux bound in the $\mathbb{R}^5$ setting (see, for instance \cite{KVNLW3}, \cite{TaoBook}).
\begin{lemma}[Energy-flux inequality] \label{lem_eflux}   For every solution $u$ to (NLW) with $(u,u_t)\in L_t^\infty(I;\dot{H}_x^{3/2}\times\dot{H}_x^{1/2})$, there exists $C_u>0$ such that we have
\begin{align}
\label{eq_eflux}
\int_{I_0}\int_{|x-y|=|t|} |u(t,x)|^{4}dS(x)dt\leq C_u\sup_{t\in I_0} |t|.
\end{align}
for every $I_0\subset I$ and $y\in\mathbb{R}^5$.
\end{lemma}
The subluminality property obtained in \cite{KVNLW3} makes essential use of the fact that the right hand side of the energy-flux inequality grows sublinearly in $|t|$ in dimension $d=3$, which is no longer the case in our setting.  To overcome this, we localize the estimate in frequency.

In the frequency localized setting, the analogue of the energy flux bound must take into account the fact that $u_{\geq N}$ satisfies 
\begin{align}
\partial_{tt}u_{\geq N}-\Delta u_{\geq N}+P_{\geq N}[u^3]=0,\label{floc_eqn}
\end{align}
in place of (NLW).  In view of this, our frequency localized form of the energy-flux bound is the following:
\begin{lemma}[Frequency localized energy-flux bound]
\label{lem2new}
Suppose that $u:\mathbb{R}\times\mathbb{R}^5\rightarrow\mathbb{R}$ is an almost periodic solution to (NLW) with $(u,u_t)\in L_t^\infty(\mathbb{R};\dot{H}_x^{3/2}\times\dot{H}_x^{1/2})$ and $\inf N(t)\geq 1$.  Then there exists $C_u>0$ such that for every $\eta\in (0,1)$ there exists $N_0=N_0(\eta)$ such that for every dyadic $N<N_0$, $x\in\mathbb{R}^5$ and compact $I_0\subset\mathbb{R}$, we have
\begin{align*}
\int_{I_0}\int_{|x-y|=|t|} |u_{\geq N}(t,y)|^4dS(y)dt\leq C_u(N^{-1}+\eta |I_0|).
\end{align*}
\end{lemma}
\begin{proof}
In the following argument, we assume that $u$ is smooth with compact support (so that integration by parts is justified); we note that this assumption can be removed by standard approximation arguments.  We work on the interval $I_0^+=I_0\cap [0,\infty)$ and note that the proof for $I_0\cap (-\infty,0]$ is similar.  Fix $N>0$ and define
\begin{align*}
\mathcal{E}(t)&=\int_{|x-y|\leq t} \frac{1}{2}|\partial_t u_{\geq N}(t,x)|^2+\frac{1}{2}|\nabla u_{\geq N}(t,x)|^2+\frac{1}{4}|u_{\geq N}(t,x)|^4dx
\end{align*}
for $t\in I_0^+$.  Differentiation in time followed by integration by parts gives
\begin{align*}
\frac{d\mathcal{E}}{dt}(t)&=\int_{|x-y|=t} \bigg[\frac{1}{2}|\partial_t u_{\geq N}(t,x)|^2+\frac{1}{2}|\nabla u_{\geq N}(t,x)|^2+\frac{1}{4}|u_{\geq N}(t,x)|^4\\
&\hspace{0.4in}+\nabla u_{\geq N}(t,x)\partial_t u_{\geq N}(t,x)\cdot \frac{x-y}{|x-y|}\bigg]dS(x)\\
&\hspace{0.2in}+\int_{|x-y|\leq t} \bigg[ \partial_t u_{\geq N}(t,x)\partial_{tt}u_{\geq N}(t,x)+[-\Delta u_{\geq N}(t,x)]\partial_t u_{\geq N}(t,x)\\
&\hspace{0.4in}+u_{\geq N}(t,x)^3\partial_t u_{\geq N}(t,x)\bigg]dx.
\end{align*}
Using ($\ref{floc_eqn}$) we get
\begin{align}
\nonumber \frac{d\mathcal{E}}{dt}(t)&=\int_{|x-y|=t} \bigg[\frac{1}{2}|\partial_t u_{\geq N}(t,x)|^2+\frac{1}{2}|\nabla u_{\geq N}(t,x)|^2+\frac{1}{4}|u_{\geq N}(t,x)|^4\\
\nonumber &\hspace{0.4in}+\nabla u_{\geq N}(t,x)\partial_t u_{\geq N}(t,x)\cdot \frac{x-y}{|x-y|}\bigg]dS(x)\\
\nonumber &\hspace{0.4in}+\int_{|x-y|\leq t} \bigg[ \bigg(u_{\geq N}(t,x)^3-P_{\geq N}[u(t,x)^3]\bigg)\partial_t u_{\geq N}(t,x)\bigg]dx\\
\nonumber &\geq \frac{1}{4}\int_{|x-y|=t} |u_{\geq N}(t,x)|^4dx\\
\label{eq1c}&\hspace{0.4in}-\bigg|\int_{|x-y|\leq t} \bigg[ \bigg(u_{\geq N}(t,x)^3-P_{\geq N}[u(t,x)^3]\bigg)\partial_t u_{\geq N}(t,x)\bigg]dx\bigg|.
\end{align}

Our next goal is to establish the bound
\begin{align}
\lVert (u_{\geq N}(t,x)^3-P_{\geq N}[u(t,x)^3])\,\partial_t u_{\geq N}(t,x)\rVert_{L_t^\infty L_x^1}&\lesssim_u \eta.\label{eq1a}
\end{align}
on the second term in ($\ref{eq1c}$), for $N$ sufficiently small.

To obtain ($\ref{eq1a}$), note that using H\"older and the Sobolev embedding gives
\begin{align}
\nonumber &\lVert (u_{\geq N}(t,x)^3-P_{\geq N}[u(t,x)^3])\,\partial_t u_{\geq N}(t,x)\rVert_{L_t^\infty L_x^1}\\
\nonumber &\hspace{0.2in}\lesssim \lVert u_{\geq N}(t,x)^3-P_{\geq N}[u(t,x)^3]\rVert_{L_t^\infty L_x^{5/3}}\lVert \partial_t u_{\geq N}(t,x)\rVert_{L_t^\infty \dot{H}_x^{1/2}}\\
&\hspace{0.2in}\lesssim_u \lVert u_{\geq N}(t,x)^3-P_{\geq N}[u_{\geq N}^3]\rVert_{L_t^\infty L_x^{5/3}}+\lVert u_{<N}\rVert_{L_x^5}\lVert u\rVert_{L_t^\infty L_x^5}^2\label{eq1b}
\end{align}
where we have used the decomposition $u=u_{<N}+u_{\geq N}$.

Now, using the almost periodicity and $\inf N(t)\geq 1$, we may choose $N$ and $N_1$ small enough so that 
\begin{align*}
N<\eta^{3/2} N_1\quad \textrm{and}\quad \lVert (u_{\leq N_1},\partial_t u_{\leq N_1})\rVert_{L_t^\infty(\dot{H}_x^{3/2}\times\dot{H}_x^{1/2})}<\eta
\end{align*}
are satisfied.  Then, writing $u_{\geq N}^3-P_{\geq N}[u_{\geq N}^3]$ as $P_{<N}[u_{\geq N}^3]$, we use the Bernstein and H\"older inequalities and the decomposition $u_{\geq N}=u_{N\leq \cdot<N_1}+u_{\geq N_1}$, followed by another instance of Bernstein's inequality and the Sobolev embedding to obtain 
\begin{align*}
(\ref{eq1b})&\lesssim_u \lVert P_{<N}[u_{\geq N}^3]\rVert_{L_t^\infty L_x^{5/3}}+\eta\\
&\lesssim_u N^{2/3}\lVert u_{\geq N}^3\rVert_{L_t^\infty L_x^{15/11}}+\eta\\
&\lesssim_u N^{2/3}\lVert u_{\geq N}\rVert_{L_t^\infty L_x^5}^2\lVert u_{\geq N}\rVert_{L_t^\infty L_x^3}+\eta\\
&\lesssim_u N^{2/3}(\lVert u_{N\leq\cdot <N_1}\rVert_{L_t^\infty L_x^{3}}+\lVert u_{\geq N_1}\rVert_{L_t^\infty L_x^3})+\eta\\
&\lesssim_u N^{2/3}(N^{-2/3}\lVert |\nabla|^{2/3}u_{N\leq\cdot <N_1}\rVert_{L_t^\infty L_x^{3}}+N_1^{-2/3}\lVert |\nabla|^{2/3}u_{\geq N_1}\rVert_{L_t^\infty L_x^3})+\eta\\
&\lesssim_u \lVert u_{N\leq\cdot <N_1}\rVert_{L_t^\infty \dot{H}_x^{3/2}}+\left(\frac{N}{N_1}\right)^{2/3}\lVert u_{\geq N_1}\rVert_{L_t^\infty \dot{H}_x^{3/2}}+\eta\\
&\lesssim_u \eta.
\end{align*}
This completes the proof of the inequality ($\ref{eq1a}$).

Combining $(\ref{eq1c}$) with ($\ref{eq1a}$), integrating over $t\in I_0$, and invoking the Fundamental Theorem of Calculus, we obtain
\begin{align*}
&\int_{I_0^+}\int_{|x-y|=t} |u_{\geq N}(t,x)|^4dxdt\\
&\hspace{0.2in}\lesssim \sup_{t\in I_0} |\mathcal{E}(t)|+\int_{I_0^+}\int_{|x-y|\leq |t|} \bigg|\bigg(u_{\geq N}(t,x)^3-P_{\geq N}[u(t,x)^3]\bigg)\partial_t u_{\geq N}(t,x)\bigg|dxdt\\
&\hspace{0.2in}\lesssim_u \sup_{t\in I_0} |\mathcal{E}(t)|+\eta |I_0^+|
\end{align*}
provided that $N$ is sufficiently small.  To bound the right hand side, we use the Bernstein inequalities, followed by H\"older and the Sobolev embedding to obtain
\begin{align*}
\mathcal{E}(t)&\lesssim N^{-1}\lVert (u,\partial_t u)\rVert^2_{L_t^\infty(\dot{H}_x^{3/2}\times\dot{H}_x^{1/2})}+N^{-1}\lVert P_{\geq N}[u(t,x)^3]\rVert_{L_x^{5/3}}\lVert \nabla u_{\geq N}\rVert_{L_x^{5/2}}\\
&\lesssim_u N^{-1}.
\end{align*}
for every $t\in I_0$, which gives the desired estimate, completing the proof of Lemma $\ref{lem2new}$.
\end{proof}

In order to invoke the frequency localized energy flux bound, Lemma $\ref{lem2new}$, in the argument for subluminality, we must adapt the bound on concentration of potential energy to account for the frequency localization.

Towards this end, in the remainder of this section, we let $u:\mathbb{R}\times\mathbb{R}^5\rightarrow\mathbb{R}$ be an almost periodic solution to (NLW) satisfying $(u,u_t)\in L_t^\infty(\mathbb{R};\dot{H}_x^{3/2}\times\dot{H}_x^{1/2})$ such that $\inf N(t)\geq 1$.

Recall that (see, e.g. \cite{KVNLW3}, \cite{BulutCubic}) there exists $C_u>0$ such that for every $k\geq 1$,
\begin{align}
\int_{I_0}\int_{|x-x(t)|\leq C/N(t)} |u(t,x)|^4dxdt\geq C_u\int_{I_0} N(t)^{-1}dt.\label{eq2aa}
\end{align}

As we noted earlier, we will need a frequency localized form of this bound.  More precisely,
\begin{lemma}[Frequency localized potential energy concentration]
\label{lem3}
Let $u:\mathbb{R}\times\mathbb{R}^5\rightarrow\mathbb{R}$ be an almost periodic solution to (NLW) satisfying $(u,u_t)\in L_t^\infty(\mathbb{R};\dot{H}_x^{3/2}\times\dot{H}_x^{1/2})$ such that $\inf N(t)\geq 1$.  Then there exists $N_0>0$ and $C_u>0$ such that for all $N\leq N_0$ and $k\geq 1$,
\begin{align}
\int_{I_0}\int_{|x-x(t)|\leq C/N(t)} |u_{\geq 8N}(t,x)|^4dxdt\geq C_u\int_{I_0} N(t)^{-1}dt.\label{eqlem3goal1}
\end{align}
\end{lemma}
Before proceeding with the proof, we recall a consequence of almost periodicity which will aid in estimating the error terms resulting from the frequency localization of the potential energy.

More precisely, let $u$ be a solution given as in Lemma $\ref{lem3}$.  Then, by the definition of almost periodicity along with the property $\inf_{t\in [0,\infty)} N(t)\geq 1$, we have
\begin{align}
\lim_{N\rightarrow 0} \left[\lVert u_{\leq N}\rVert_{L_t^\infty\dot{H}_x^{3/2}}+\lVert P_{\leq N}u_t\rVert_{L_t^\infty\dot{H}_x^{1/2}}\right]=0.\label{eqrem1}
\end{align}

\begin{proof}[Proof of Lemma $\ref{lem3}$.]
Fix $N>0$.  Applying ($\ref{eq2aa}$) to $u_{\geq {8N}}$, we obtain
\begin{align}
\nonumber &\left[\int_{I_0}\int_{|x-x(t)|\leq C/N(t)} |u_{\geq 8N}(t,x)|^4dxdt\right]^{1/4}\\
\nonumber &\hspace{0.2in}\geq \left[\int_{I_0}\int_{|x-x(t)|\leq C/N(t)} |u(t,x)|^4dxdt\right]^{1/4}-\left[\int_{I_0}\int_{|x-x(t)|\leq C/N(t)} |u_{<8N}(t,x)|^4dxdt\right]^{1/4}\\
&\hspace{0.2in}\gtrsim_u C_u\int_{I_0} N(t)^{-1}dt-\left[\int_{I_0}\int_{|x-x(t)|\leq C/N(t)} |u_{<8N}(t,x)|^4dxdt\right]^{1/4}.\label{eq3}
\end{align}

It therefore suffices to estimate the space-time norm appearing in ($\ref{eq3}$).  For this, we fix $\eta\in(0,1)$ to be determined later in the argument, and use H\"older's inequality followed by the decomposition $u=u_{<8N}+u_{\geq 8N}$ to obtain
\begin{align}
\nonumber \int_{I_0}\int_{|x-x(t)|\leq C/N(t)} |u_{<8N}(t,x)|^4dxdt&\leq \int_{I_0}\frac{C}{N(t)}\lVert u_{<8N}(t,x)\rVert_{L_x^{5}}^4dt\\
&\lesssim_u \eta \int_{I_0} N(t)^{-1}dt,\label{eq4}
\end{align}
where we have used ($\ref{eqrem1}$) to obtain the second inequality.

Substituting ($\ref{eq4}$) into ($\ref{eq3}$), we then have,
\begin{align*}
\int_{I_0}\int_{|x-x(t)|\leq C/N(t)} |u_{\geq 8N}(t,x)|^4dxdt&\geq (C_u-\eta C'_u)\int_{I_0} N(t)^{-1}dt.
\end{align*}
Choosing $\eta$ small enough so that $\eta C'_u<C_u/2$ yields the bound ($\ref{eqlem3goal1}$).
\end{proof}

Having obtained the frequency localized energy flux bound and potential energy concentration we now turn to the proof of the subluminality property.  For this purpose, we recall the following lemma from \cite{KVNLW3}, which states that the function $x(t)$ can be chosen to have a certain centering property and establishes a first relationship between the speed of $x(t)$ and the frequency scale function $N(t)$.
\begin{lemma}\cite{KVNLW3}
\label{lem2a}
Suppose that $u:\mathbb{R}\times\mathbb{R}^5\rightarrow\mathbb{R}$ is an almost periodic solution to (NLW) satisfying $(u,u_t)\in L_t^\infty(\mathbb{R};\dot{H}_x^{3/2}\times\dot{H}_x^{1/2})$ such that $\inf N(t)\geq 1$.  

Then there exists a constant $C_u>0$ such that $x(t)$ can be chosen to satisfy
\begin{align*}
\inf_{\omega\in S^4} \int_{\omega\cdot (x-x(t))>0} ||\nabla|^{3/2}u(t,x)|^2+||\nabla|^{1/2}u_t(t,x)|^2dx\geq \frac{1}{C_u}.
\end{align*}

Moreover, there exists $c=c_u\in (0,1)$ such that for all $\tau_1,\tau_2\in\mathbb{R}$ satisfying $N(\tau_1)\leq N(\tau_2)$, we have 
\begin{align*}
|x(\tau_1)-x(\tau_2)|\geq |\tau_1-\tau_2|-cN(\tau_1)^{-1}\quad\Rightarrow\quad N(\tau_1)\geq c^2N(\tau_2).
\end{align*}
\end{lemma}
The arguments used to obtain the claims in Lemma $\ref{lem2a}$ follow from the small data theory and the finite speed of propagation; in particular the discussion of Proposition $4.1$ and Lemma $4.4$ in \cite{KVNLW3} applies equally well to the current setting.

We are now ready to prove the main result of this section,
\begin{theorem}[Subluminality]
\label{thm_sublum}
Suppose that $u:\mathbb{R}\times\mathbb{R}^5\rightarrow\mathbb{R}$ is an almost periodic solution to (NLW) satisfying $(u,u_t)\in L_t^\infty(\dot{H}_x^{3/2}\times\dot{H}_x^{1/2})$ and such that $\inf_{t\in \mathbb{R}} N(t)\geq 1$.  Then there exists $\delta>0$ such that for every $t,\tau\in \mathbb{R}$ with $|t-\tau|\geq \delta^{-1}$, we have
\begin{align*}
|x(t)-x(\tau)|\leq (1-\delta)|t-\tau|.
\end{align*}
\end{theorem}
\begin{proof} 
As in \cite{KVNLW3}, we note that it suffices to show the following claim: there exists $A>1$ such that for every $t_0\in\mathbb{R}$ there exists $t\in [t_0,t_0+AN(t_0)^{-1}]$ such that
\begin{align}
\label{goalA1}|x(t)-x(t_0)|\leq |t-t_0|-A^{-1}N(t_0)^{-1}.
\end{align}

Suppose to the contrary that ($\ref{goalA1}$) failed.  Then for every $A>0$ there exists $t_0=t_0(A)\in\mathbb{R}$ such that
\begin{align}
|x(t)-x(t_0)|>|t-t_0|-A^{-1}N(t_0)^{-1}\label{eaa1}
\end{align}
on $[t_0,t_0+AN(t_0)^{-1}]$.

Let $c$ be as in Lemma $\ref{lem2a}$, fix $A>c^{-1}$, and choose $t_0$ as in ($\ref{eaa1}$).  As a first step, we show 
\begin{align}
c^2N(t_0)\leq N(t)\leq c^{-2}N(t_0)\quad \textrm{for}\quad t\in [t_0,t_0+AN(t_0)^{-1}].\label{eaa2}
\end{align}

To obtain ($\ref{eaa2}$), suppose first that $t\in [t_0,t_0+AN(t_0)^{-1}]$ satisfies $N(t)\leq N(t_0)$.  By ($\ref{eaa1}$), we then have the inequality
$|x(t)-x(t_0)|>|t-t_0|-cN(t_0)^{-1}\geq |t-t_0|-cN(t)^{-1}$,
so that Lemma $\ref{lem2a}$ (ii) implies $c^2N(t_0)\leq N(t)$.  Moreover, the bound $N(t)\leq c^{-2}N(t_0)$ follows trivially from $c<1$, verifying ($\ref{eaa2}$) in this case.  On the other hand, when $N(t)>N(t_0)$, the bound $c^2N(t_0)\leq N(t)$ is immediate, and observing that ($\ref{eaa1}$) implies $|x(t)-x(t_0)|>|t-t_0|-cN(t_0)^{-1}$, we invoke Lemma $\ref{lem2a}$ (ii) again to obtain $N(t)\leq c^{-2}N(t_0)$ as desired.  Thus, ($\ref{eaa2}$) holds.

Having obtained ($\ref{eaa2}$), we now continue towards obtaining the desired contradiction to prove the theorem.  Note that by space and time translation symmetries it suffices to assume $t_0=0$ and $x(t_0)=0$.  We will obtain the desired contradiction by obtaining incompatible lower and upper bounds on 
\begin{align}
\int_{\beta N(0)^{-1}}^{AN(0)^{-1}}\int_{||x|-t|\leq \beta N(0)^{-1}} |u_{\geq 8N}(t,x)|^4dxdt.\label{eaa3}
\end{align}
for a suitable value of $\beta>0$ to be chosen later in the argument.  

We first obtain the lower bound on ($\ref{eaa3}$).  	Toward this end, we first obtain the set inclusion
\begin{align}
\{x:|x-x(t)|\leq C/N(t)\}\subset \{x:\big||x|-t\big|\leq \beta N(0)^{-1}\}.\label{eqinclu}
\end{align}

To show ($\ref{eqinclu}$), we note that for $t_0=0$ and $x(t_0)=0$, ($\ref{eaa1}$) becomes 
\begin{align*}
|x(t)|>|t|-A^{-1}N(0)^{-1}>|t|-cN(0)^{-1}\quad\textrm{for}\quad 0\leq t\leq AN(0)^{-1}.
\end{align*}
On the other hand, ($\ref{fsprop_xt}$) and ($\ref{eaa2}$) imply
\begin{align*}
|x(t)|\leq |t|+C(N(t)^{-1}+N(0)^{-1})\leq |t|+CN(0)^{-1},
\end{align*}
and we therefore conclude that for all $t\in [0,AN(0)^{-1}]$,
\begin{align}
\big||x(t)|-t\big|\lesssim_u N(0)^{-1}.\label{five13}
\end{align}
We now use ($\ref{five13}$) to observe that the inequality $|x-x(t)|\leq C/N(t)$ gives
\begin{align*}
\big||x|-t\big|&=\big||x|-|x(t)|\big|+\big||x(t)|-t\big|\\
&\leq |x-x(t)|+\big||x(t)|-t\big|\\
&\lesssim_u N(t)^{-1}+N(0)^{-1}\\
&\lesssim_u (c^{-2}+1)N(0)^{-1},
\end{align*}
where the implied constant is independent of $A$.  We may therefore conclude ($\ref{eqinclu}$) when $\beta$ is chosen sufficiently large (independent of $A$).  Then, invoking Lemma $\ref{lem3}$, we obtain 
\begin{align}
\nonumber (\ref{eaa3})&\geq \int_{\beta N(0)^{-1}}^{AN(0)^{-1}}\int_{|x-x(t)|\leq C/N(t)} |u_{\geq 8N}(t,x)|^4dxdt\\
\nonumber &\gtrsim_u \int_{\beta N(0)^{-1}}^{AN(0)^{-1}}N(t)^{-1}dt\\
&\gtrsim_u c^2(A-\beta)N(0)^{-2}.\label{eaaa2} 
\end{align}
which establishes the desired lower bound on ($\ref{eaa3}$),

We now establish the upper bound on ($\ref{eaa3}$).  Using Lemma $\ref{lem2new}$ followed by the change of variables $z=x-y$ in the $x$ variable, Fubini's theorem, and the change of variables $y'=z+y$ in the $y$ variable, we obtain
\begin{align*}
&[2\beta N(0)^{-1}]^5(N^{-1}+\eta (A-\beta)N(0)^{-1})\\
&\hspace{0.2in}\gtrsim_u \int_{|y|\leq 2\beta N(0)^{-1}}\int_{\beta N(0)^{-1}}^{AN(0)^{-1}}\int_{|x-y|=t} |u_{\geq N}(t,x)|^4dS(x)dtdy\\
&\hspace{0.2in}\gtrsim_u \int_{\beta N(0)^{-1}}^{AN(0)^{-1}}\int_{|z|=t}\int_{|y'-z|\leq 2\beta N(0)^{-1}} |u_{\geq N}(t,x)|^4dy'dS(z)dt.
\end{align*}
for $N$ sufficiently small.
Observing the inequality
\begin{align*}
\int_{|z|=t} \chi_{\{z:|z-y'|\leq 2\beta N(0)^{-1}\}}(z)dS(z)\gtrsim [\beta N(0)^{-1}]^4
\end{align*}
for $t\geq \beta N(0)^{-1}$ and $y'\in \{||y'|-t|\leq \beta N(0)^{-1}\}$, we therefore get
\begin{align}
(\ref{eaa3})&\lesssim_u 2^5\beta N(0)^{-1}N^{-1}+\eta 2^5\beta(A-\beta)N(0)^{-2}\label{eq2a}
\end{align}
Combining the upper and lower bounds ($\ref{eq2a}$) and $(\ref{eaaa2})$ for $(\ref{eaa3})$ and choosing $\eta$ sufficiently small (depending on the implicit constant and $\beta$), we then have
\begin{align}
(A-\beta)N(0)^{-2}\lesssim_u 2^5\beta N(0)^{-1}N^{-1}.\label{eqa111}
\end{align}
Note that the implicit constant in $(\ref{eqa111})$ depends only on the a priori bound $\lVert (u,u_t)\rVert_{L_t^\infty(\dot{H}_x^{3/2}\times\dot{H}_x^{1/2})}$ and the compactness modulus function, which are invariant under scaling.  In particular, using the scaling of the equation we may assume that $N(0)=1$.  Letting $A\rightarrow\infty$ then gives the desired contradiction.
\end{proof}

\section{Improved decay properties}
\label{sec_regularity}

The main goal of this section is to complete the proof of the following theorem, which states that the soliton-like and low-to-high frequency cascade solutions have finite energy.
\begin{theorem}
\label{thm_negreg}
Assume that $u:\mathbb{R}\times\mathbb{R}^5\rightarrow\mathbb{R}$ is an almost periodic solution to (NLW) with $(u,u_t)\in L_t^\infty(\mathbb{R};\dot{H}_x^{3/2}\times \dot{H}_x^{1/2})$ and 
\begin{align*}
\inf_{t\in I} N(t)\geq 1.
\end{align*}
Then $(u,u_t)\in L_t^\infty(\mathbb{R};\dot{H}_x^{1}\times L_x^2)$ and, moreover, there exist $\beta>0$ and $N_1>0$ such that 
\begin{align*}
\lVert \langle x-x(t)\rangle^\beta P_{\leq N_1}\nabla u\rVert_{L_t^\infty L_x^2}+\lVert \langle x-x(t)\rangle^\beta P_{\leq N_1}u_t\rVert_{L_t^\infty L_x^2}<\infty.
\end{align*}
\end{theorem}
Theorem $\ref{thm_negreg}$ will be proved by establishing a slight improvement of the decay properties of $u$, in which bounds of the form $(u,u_t)\in L_t^\infty(\dot{H}_x^s\times\dot{H}_x^{s-1})$ are shown to imply $(u,u_t)\in L_t^\infty(\dot{H}_x^{s-\epsilon}\times\dot{H}_x^{s-1-\epsilon})$ for $\epsilon>0$ sufficiently small (see Lemma $\ref{lem_somenegreg}$ below).  Theorem $\ref{thm_negreg}$ is then proven by an iterative application of Lemma $\ref{lem_somenegreg}$, starting with the a priori bound $(u,u_t)\in L_t^\infty(\dot{H}_x^{s_c}\times\dot{H}_x^{s_c-1})$.

More precisely, we obtain 
\begin{lemma}
\label{lem_somenegreg}
Assume that $u:\mathbb{R}\times\mathbb{R}^5\rightarrow\mathbb{R}$ is a global almost periodic solution to (NLW) with $(u,u_t)\in L_t^\infty(\mathbb{R};\dot{H}_x^{3/2}\times\dot{H}_x^{1/2})$ and
\begin{align*}
\inf_{t\in I} N(t)\geq 1.
\end{align*}

Then there exist constants $\epsilon_0,\beta>0$ and a dyadic number $N_1>0$ such that for any $s\in (1,\frac{3}{2}]$, the condition
\begin{align*}
(u,u_t)\in L_t^\infty(\mathbb{R};\dot{H}_x^{s}\times \dot{H}_x^{s-1})
\end{align*}
implies 
\begin{align}
\label{plus1}\lVert \langle x-x(t)\rangle^\beta P_{\leq N_1}|\nabla|^{s-\epsilon} u\rVert_{L_t^\infty L_x^2}+\lVert \langle x-x(t)\rangle^\beta P_{\leq N_1}|\nabla|^{s-1-\epsilon}u_t\rVert_{L_t^\infty L_x^2}<\infty
\end{align}
for every $0<\epsilon<\epsilon_0$ satisfying $s-\epsilon\geq 1$.

In particular, this allows us to conclude
\begin{align}\label{resthmnegreg}
(u,u_t)\in L_t^\infty(\mathbb{R};\dot{H}_x^{s-\epsilon}\times \dot{H}_x^{s-1-\epsilon}) 
\end{align}
for any such $\epsilon$.
\end{lemma}

\begin{proof}
As in \cite{KVNLW3}, we begin by observing that to prove the claim it suffices to obtain an estimate of the form
\begin{align}
\label{plus3}\lVert P_{\leq N_1}|\nabla|^{s-\epsilon}u(0)\rVert_{L_x^2(B_R)}+\lVert P_{\leq N_1}|\nabla|^{s-1-\epsilon} u_t(0)\rVert_{L_x^2(B_R)}\lesssim_u R^{-\beta}
\end{align}
for some $\beta>0$ and all balls $B_R=\{x:|x-x_0|<R\}$ with $x_0\in\mathbb{R}^5$ and $|x_0|=3R$.  

Indeed, using the time and space translation symmetries to reduce the desired estimate ($\ref{plus1}$) to the case $t=0$, $x(0)=0$, a covering argument by Whitney balls (that is, balls of the form $B_R$) shows that ($\ref{plus1}$) follows from ($\ref{plus3}$).

To see that ($\ref{resthmnegreg}$) follows from ($\ref{plus1}$), fix $\epsilon\in (0,s-1]$.  Decomposing the solution into low and high frequencies, and using Bernstein's inequality, we obtain
\begin{align}
\nonumber &\lVert |\nabla|^{s-\epsilon} u\rVert_{L_t^\infty L_x^2}+\lVert |\nabla|^{s-1-\epsilon} u_t\rVert_{L_t^\infty L_x^2}\\
\nonumber &\hspace{0.2in}\leq \lVert |\nabla|^{s-\epsilon} P_{\leq N_1}u\rVert_{L_t^\infty L_x^2}+\lVert |\nabla|^{s-1-\epsilon}\partial_t P_{\leq N_1}u\rVert_{L_t^\infty L_x^2}\\
\nonumber &\hspace{0.4in}+\sum_{N>N_1} \lVert |\nabla|^{s-\epsilon} u_N\rVert_{L_t^\infty L_x^2}+\lVert |\nabla|^{s-1-\epsilon} \partial_t u_N\rVert_{L_t^\infty L_x^2}\\
\nonumber &\hspace{0.2in}\lesssim \lVert |\nabla|^{s-\epsilon} P_{\leq N_1}u\rVert_{L_t^\infty L_x^2}+\lVert |\nabla|^{s-1-\epsilon}\partial_t P_{\leq N_1}u\rVert_{L_t^\infty L_x^2}+\sum_{N>N_1}N^{s-\epsilon-\frac{3}{2}}\\
\label{eq_nreg}&\hspace{0.2in}\lesssim \lVert |\nabla|^{s-\epsilon} P_{\leq N_1}u\rVert_{L_t^\infty L_x^2}+\lVert |\nabla|^{s-1-\epsilon}\partial_t P_{\leq N_1}u\rVert_{L_t^\infty L_x^2}+1,
\end{align}
where to obtain the second inequality we have used the a priori bound $(u,u_t)\in L_t^\infty(\dot{H}_x^{3/2}\times\dot{H}_x^{1/2})$.

Thus, in order to obtain ($\ref{plus1}$) and ($\ref{resthmnegreg}$), it suffices to show ($\ref{plus3}$).

Recalling the definition of $\theta$ in ($\ref{eq_theta}$), we now choose $0<\eta<2$ such that $\theta(\xi)\geq \frac{1}{2}$ for all $|\xi|<\eta$, and a dyadic number $N_1$ such that $N_1\leq \frac{\eta}{2}$.  Let $\phi_R$ be given by
\begin{align*}
\phi_R(x)=\phi\bigg(\frac{x-x_0}{R}\bigg),
\end{align*}
and note that $N_1\leq \frac{\eta}{2}$ gives $\supp \widehat{P_{\leq N_1}f}\subset \{\xi:|\xi|\leq \eta\}$ for $f\in L^2$, and thus
\begin{align*}
\lVert P_{\leq N_1}f\rVert_{L_x^2}\lesssim \lVert \theta(i\nabla)P_{\leq N_1}f\rVert_{L_x^2}.
\end{align*}

Using the Duhamel formula of Proposition $\ref{prop_duhamel}$ forward and backward in time, we obtain
\begin{align}
\nonumber &\lVert P_{\leq N_1}|\nabla|^{s-\epsilon}u(0)\rVert_{L_x^2(B_R)}+\lVert P_{\leq N_1}|\nabla|^{s-1-\epsilon} u_t(0)\rVert_{L_x^2(B_R)}\\
\nonumber &\hspace{0.2in}\lesssim \lVert \theta(i\nabla)P_{\leq N_1}|\nabla|^{s-\epsilon}u(0)\rVert_{L_x^2(B_R)}^2+\lVert \theta(i\nabla)P_{\leq N_1}|\nabla|^{s-1-\epsilon}\partial_t u(0)\rVert_{L_x^2(B_R)}^2\\
\nonumber &\hspace{0.2in}\lesssim -\bigg\langle \int_0^\infty\nabla \frac{\sin(-t|\nabla|)}{|\nabla|}\theta(i\nabla)|\nabla|^{s-1-\epsilon}P_{\leq N_1}F(u(t))dt,\\
\nonumber &\hspace{1.25in}\phi_R\int_{-\infty}^0\nabla \frac{\sin(-\tau|\nabla|)}{|\nabla|}\theta(i\nabla)|\nabla|^{s-1-\epsilon}P_{\leq N_1}F(u(\tau))d\tau\bigg\rangle\\
\nonumber &\hspace{0.4in}- \bigg\langle \int_0^\infty \cos(-t|\nabla|)\theta(i\nabla)|\nabla|^{s-1-\epsilon}P_{\leq N_1}F(u(t))dt,\\
&\hspace{1.25in}\phi_R\int_{-\infty}^0\cos(-\tau|\nabla|)\theta(i\nabla)|\nabla|^{s-1-\epsilon}P_{\leq N_1}F(u(\tau))d\tau\bigg\rangle\label{eqabcdef2}
\end{align}

In order to obtain bounds for the space-time integrals in ($\ref{eqabcdef2}$), following \cite{KVNLW3}, we now define the cutoff functions $\rho_R$, $\sigma_R$, $\tilde{\rho_R}$ and $\tilde{\sigma_R}$.
\begin{lemma}[Specification of the cutoffs]
\label{lem_cutoffs}
There exists $R_0>1$ such that for every $R>R_0$ there exist cutoff functions $\rho_R, \sigma_R$ and $\tilde{\rho_R}, \tilde{\sigma_{R}}\in C^\infty(\mathbb{R}\times\mathbb{R}^5;[0,1])$ such that 
\begin{enumerate}
\item for all $(t,x)\in \mathbb{R}\times \supp\phi_R$,
\begin{align*}
&\frac{\sin(-t|\nabla|)}{|\nabla|}\theta(i\nabla)|\nabla|^{s-1-\epsilon}P_{\leq N_1}F(u(t))\\
&\hspace{0.6in}=\frac{\sin(-t|\nabla|)}{|\nabla|}\theta(i\nabla)\rho_R(t,x)|\nabla|^{s-1-\epsilon}P_{\leq N_1}F(u(t))
\end{align*}
and
\begin{align*}
&\cos(-t|\nabla|)\theta(i\nabla)|\nabla|^{s-1-\epsilon}P_{\leq N_1}F(u(t))\\
&\hspace{0.6in}=\cos(-t|\nabla|)\theta(i\nabla)\rho_R(t,x)|\nabla|^{s-1-\epsilon}P_{\leq N_1}F(u(t)),
\end{align*}
\item for all $(t,x)\in\{(t,x):t<\frac{R}{2}$ or $t>\frac{10}{\delta}R\}$,
\begin{align*}
\sigma_R(t,x)=1,
\end{align*}
\item there exists $C_1>0$ such that for all $t\in\mathbb{R}$ and $\delta$ as in Theorem $\ref{thm_sublum}$,
\begin{align*}
\dist (\supp \rho_R\sigma_R(t),\supp (1-\tilde{\rho_R}\tilde{\sigma_R}(t)))\geq C_1(|t|+R)
\end{align*}
and
\begin{align*}
\dist(x(t),\supp(\tilde{\sigma_R}\tilde{\rho_R}(t)))\geq C_1(|t|+R),
\end{align*}
\item there exists $C_2>0$ such that
\begin{align*}
\lVert \nabla \sqrt{\rho_R\sigma_R}\rVert_{L_t^\infty L_x^5}\leq C_2,\quad\textrm{and}\quad\lVert \nabla [\rho_R(1-\sigma_R)]\rVert_{L_t^\infty L_x^5}\leq C_2,
\end{align*}
\item there exists $C_3>0$ such that for all $\tau<0<t$ and $x,y\in\mathbb{R}^5$ with $(t,x),(\tau,y)\in \supp (1-\sigma_R)$,
\begin{align*}
|t|+|\tau|+|x|+|y|\leq C_3R\quad\textrm{and}\quad |t-\tau|-|x-y|\geq C_3R.
\end{align*}
\end{enumerate}
\end{lemma}
\begin{proof}
Fix $R_0\geq 1$ to be determined later in the argument, and let $R>R_0$ be given.  Recall that $x_0\in\mathbb{R}^5$ is an arbitrary point satisfying $|x_0|=3R$.  

We begin by choosing $\rho_R$ such that $\rho_R(t,x)=1$ on $\{(t,x)\in \mathbb{R}\times\mathbb{R}^5:\big||x-x_0|-|t|\big|\leq \frac{6}{5}R\}$ and
\begin{align*}
\supp \rho_R\subset \{(t,x):(1-\tfrac{\delta}{10^6})|t|-\frac{6}{5}R\leq |x-x_0|\leq (1+\tfrac{\delta}{10^6})|t|+\tfrac{6}{5}R\},
\end{align*}
along with the condition that for each multi-index $\alpha=(\alpha_1,\cdots,\alpha_5)$, there exists $C_{\alpha}>0$ such that
\begin{align*}
|\partial_x^\alpha \rho_R|\leq C_\alpha(|t|+R)^{-|\alpha|}.
\end{align*}

Similarly, we choose $\sigma_R$ such that $\sigma_R=1$ on $\{(t,x): t\in (-\infty,\frac{R}{2})\cup (\frac{10}{\delta}R,\infty)\,\textrm{or}\,|x-x(t)|> \frac{\delta}{5}(|t|+R)\}$ and
\begin{align*}
\supp \sigma_R\subset\{(t,x): t\in (-\infty,\tfrac{R}{2})\cup (\tfrac{10}{\delta}R,\infty)\,\textrm{or}\,|x-x(t)|>\tfrac{\delta}{10}(|t|+R)\}
\end{align*}
along with the derivative bound
\begin{align*}
|\partial_x^\alpha \sigma_R|+|\partial_x^\alpha (1-\sigma_R)|\leq C_\alpha(|t|+R)^{-|\alpha|}
\end{align*}
for each multi-index $\alpha$.

Next, we choose $\tilde{\rho}_R$ such that $\tilde{\rho}_R=1$ on $\{(t,x):\dist(x,\{x:x\in \supp_x \rho_R(t,x)\})\leq \frac{1}{10}R+\frac{\delta}{10^6}|t|\}$, where $\supp_x$ denotes the support in $x$, and
\begin{align*}
\supp \tilde{\rho}_R\subset \{(t,x):(1-\tfrac{3\delta}{10^6})|t|-\tfrac{8}{5}R\leq |x-x_0|\leq (1+\tfrac{3\delta}{10^6})|t|+\tfrac{8}{5}R\}
\end{align*}
along with the derivative bound
\begin{align}
\label{dbound1} |\partial_x^\alpha \tilde{\rho}_R|\leq C_\alpha(|t|+R)^{-|\alpha|}
\end{align}
for each multi-index $\alpha$.

To finish the construction, we now specify $\tilde{\sigma}_R$.  In particular, we choose $\tilde{\sigma}_R$ such that $\tilde{\sigma}_R=1$ on $\{(t,x):\dist(x,\supp_x \sigma_R(t))\leq \frac{\delta}{40}(|t|+R)\}$ and
\begin{align*}
\supp_x \tilde{\sigma}_R(t)\subset \{x:|x-x(t)|\geq \tfrac{\delta}{20}(|t|+R)\}
\end{align*}
for each $t\in [\frac{R}{2},\frac{10}{\delta}R]$, as well as the derivative bounds
\begin{align}
\label{dbound2} |\partial_x^\alpha \tilde{\sigma}_R|\leq C_\alpha (|t|+R)^{-|\alpha|}
\end{align}
for each multi-index $\alpha$.

The properties (i)-(v) are now easily verified: noting that $\supp \widehat{\theta}\subset \{x:|x_i|\leq 4, i=1,\cdots,5\}$, (i) follows from the representations ($\ref{lop1}$) and ($\ref{lop2}$) for the linear propagator (in particular, this is a formulation of Huygens' principle).  Properties (ii) and (iii) then follow directly from the construction, while (iv) follows from the derivative bounds given in the construction of each function, after a suitable change of variables.

To conclude (v), we note $\supp (1-\sigma_R)\subset \{(t,x):|t|\in [\frac{R}{2},\frac{10}{\delta}R]\}$, and thus there exists $C>0$ such that $|t|+|\tau|\leq CR$ for $(t,x),(\tau,y)\in \supp (1-\sigma_R)$.  
The finite speed of propagation in the form of Lemma $\ref{lem_a2}$ then implies
\begin{align*}
|x(t)|\leq |x(t)-x(0)|+|x(0)|\leq |t|+|x(0)|+C.
\end{align*}
for all $t\in \mathbb{R}$, and thus for all $(t,x)\in \supp (1-\sigma_R)$ we have
\begin{align*}
|x|&\leq |x-x(t)|+|x(t)|\leq \tfrac{\delta}{5}(|t|+R)+|t|+C\leq CR
\end{align*}
for some $C>0$, where we have used the facts $R\geq 1$ and $\sigma_R=1$ on $\{(t,x):|x-x(t)|>\frac{\delta}{5}(|t|+R)\}$.  We therefore conclude that there exists $C>0$ such that for all $(t,x),(\tau,y)\in \supp (1-\sigma_R)$,
\begin{align*}
|t|+|\tau|+|x|+|y|\leq CR
\end{align*}
which is the first component of (v).

To see that the second component holds, we fix $R_0>\max\{1,\delta^{-1}\}$ and apply Theorem $\ref{thm_sublum}$ to obtain
\begin{align*}
|t-\tau|-|x-y|&\geq \delta|t-\tau|-|x-x(t)|-|x(\tau)-y|\\
&\geq \tfrac{2\delta}{5}R,
\end{align*}
where we observe that $(t,x),(\tau,y)\in \supp (1-\sigma_R)$ with $\tau<0<t$ gives the bounds $|t-\tau|\geq R$, $|x-x(t)|\leq \frac{\delta}{5}(t+R)$ and $|y-x(\tau)|\leq \frac{\delta}{5}(-\tau+R)$.

This completes the proof of Lemma $\ref{lem_cutoffs}$.
\end{proof}

Having specified the cutoffs, we now use these functions to decompose the spatial integration given by the inner product in ($\ref{eqabcdef2}$), obtaining
\begin{align}
\nonumber (\ref{eqabcdef2})&\lesssim -\langle \int_0^\infty \nabla \frac{\sin(-t|\nabla|)}{|\nabla|}\theta(i\nabla)\rho_R(t,x)|\nabla|^{s-1-\epsilon}P_{\leq N_1}F(u(t)),\\
\nonumber &\hspace{0.8in}\phi_{R}\int_{-\infty}^0\nabla \frac{\sin(-\tau|\nabla|)}{|\nabla|}\theta(i\nabla)\rho_R(t,x)|\nabla|^{s-1-\epsilon}P_{\leq N_1}F(u(\tau))\rangle\\
\nonumber &\hspace{0.2in}- \langle \int_0^\infty \cos(-t|\nabla|)\theta(i\nabla)\rho_R(t,x)|\nabla|^{s-1-\epsilon}P_{\leq N_1}F(u(t)),\\
\nonumber &\hspace{0.8in}\phi_{R}\int_{-\infty}^0\cos(-\tau|\nabla|)\theta(i\nabla)\rho_R(t,x)|\nabla|^{s-1-\epsilon}P_{\leq N_1}F(u(\tau))\rangle\\
\nonumber &=-\langle A_1^+,\phi_{R}A_1^-\rangle-\langle A_1^+,\phi_{R}A_2^-\rangle-\langle A_2^+,\phi_{R}A_1^-\rangle-\langle A_2^+,\phi_{R}A_2^-\rangle\\
\nonumber &\hspace{0.4in}-\langle B_1^+,\phi_{R}B_1^-\rangle-\langle B_1^+,\phi_{R}B_2^-\rangle-\langle B_2^+,\phi_{R}B_1^-\rangle-\langle B_2^+,\phi_{R}B_2^-\rangle\\
\nonumber &\leq \lVert A_1^+\rVert_{L_x^2}\lVert A_1^-\rVert_{L_x^2}+\lVert A_1^+\rVert_{L_x^2}\lVert A_2^-\rVert_{L_x^2}+\lVert A_2^+\rVert_{L_x^2}\lVert A_1^-\rVert_{L_x^2}\\
\nonumber &\hspace{0.4in}+\lVert B_1^+\rVert_{L_x^2}\lVert B_1^-\rVert_{L_x^2}+\lVert B_1^+\rVert_{L_x^2}\lVert B_2^-\rVert_{L_x^2}+\lVert B_2^+\rVert_{L_x^2}\lVert B_1^-\rVert_{L_x^2}\\
&\hspace{0.4in}+|\langle A_2^+,\phi_{R}A_2^-\rangle|+|\langle B_2^+,\phi_{R}B_2^-\rangle|\label{ceq3}
\end{align}
where we have set
\begin{align*}
A_1^+(x)&:=\int_0^\infty \nabla\frac{\sin(-t|\nabla|)}{|\nabla|}\theta(i\nabla)\rho_R(t,x)\sigma_R(t,x)|\nabla|^{s-1-\epsilon}P_{\leq N_1}F(u(t))dt,\\
A_1^-(x)&:=\int_{-\infty}^0\nabla \frac{\sin(-\tau|\nabla|)}{|\nabla|}\theta(i\nabla)\rho_R(\tau,x)\sigma_R(\tau,x)|\nabla|^{s-1-\epsilon}P_{\leq N_1}F(u(\tau))d\tau,\\
A_2^+(x)&:=\int_0^\infty \nabla\frac{\sin(-t|\nabla|)}{|\nabla|}\theta(i\nabla)\rho_R(t,x)(1-\sigma_R(t,x))|\nabla|^{s-1-\epsilon}P_{\leq N_1}F(u(t))dt,\\
A_2^-(x)&:=\int_{-\infty}^0\nabla \frac{\sin(-\tau|\nabla|)}{|\nabla|}\theta(i\nabla)\rho_R(\tau,x)(1-\sigma_R(\tau,x))|\nabla|^{s-1-\epsilon}P_{\leq N_1}F(u(\tau))d\tau,
\end{align*}
as well as
\begin{align*}
B_1^+(x)&:=\int_0^\infty \cos(-t|\nabla|)\theta(i\nabla)\rho_R(t,x)\sigma_R(t,x)|\nabla|^{s-1-\epsilon}P_{\leq N_1}F(u(t))dt,\\
B_1^-(x)&:=\int_{-\infty}^0\cos(-\tau|\nabla|)\theta(i\nabla)\rho_R(\tau,x)\sigma_R(\tau,x)|\nabla|^{s-1-\epsilon}P_{\leq N_1}F(u(\tau))d\tau,\\
B_2^+(x)&:=\int_0^\infty \cos(-t|\nabla|)\theta(i\nabla)\rho_R(t,x)(1-\sigma_R(t,x))|\nabla|^{s-1-\epsilon}P_{\leq N_1}F(u(t))dt,\\
B_2^-(x)&:=\int_{-\infty}^0\cos(-\tau|\nabla|)\theta(i\nabla)\rho_R(\tau,x)(1-\sigma_R(\tau,x))|\nabla|^{s-1-\epsilon}P_{\leq N_1}F(u(\tau))d\tau.
\end{align*}

The remainder of the proof of Lemma $\ref{lem_somenegreg}$ now consists of estimating the norms and inner products appearing in ($\ref{ceq3}$).  We accomplish this in the following three propositions.
\begin{proposition}
\label{lem53}
We have
\begin{align*}
\max \{\lVert A_1^+\rVert_{L_x^2},\lVert A_1^-\rVert_{L_x^2},\lVert B_1^+\rVert_{L_x^2},\lVert B_1^-\rVert_{L_x^2}\}\lesssim R^{-1/2-\beta}
\end{align*}
for every $\beta<\frac{1}{1000}$, provided that $\epsilon>0$ is sufficiently small.
\end{proposition}

\begin{proof}
We argue as in \cite{KVNLW3}.  We show the estimate for $\lVert A_1^+\rVert_{L_x^2}$ and remark that the other estimates are similar.  Using the inhomogeneous Strichartz inequality and the fractional product rule followed by H\"older's inequality, we obtain
\begin{align}
\nonumber &\lVert A_1^+\rVert_{L_x^2}\\
\nonumber &\hspace{0.2in}\leq \lim_{T\rightarrow\infty} \bigg\lVert \int_0^T |\nabla|^{5/4}\frac{\sin(-t|\nabla|)}{|\nabla|}\theta(i\nabla)\rho_R(t,x)\sigma_R(t,x)|\nabla|^{s-\frac{5}{4}-\epsilon}P_{\leq N_1}F(u(t))dt\bigg\rVert_{L_x^2}\\
\nonumber &\hspace{0.2in}\lesssim \left\lVert \nabla \big[\rho_R(t,x)\sigma_R(t,x)|\nabla|^{s-\frac{5}{4}-\epsilon}P_{\leq N_1}F(u(t))\big]\right\rVert_{L_t^{2}L_x^{4/3}(\mathbb{R}\times\mathbb{R}^5)},\\
\nonumber &\hspace{0.2in}\lesssim \left\lVert \big[\nabla[\rho_R\sigma_R]\big]|\nabla|^{s-\frac{5}{4}-\epsilon}P_{\leq N_1}F(u(t))\right\rVert_{L_t^2L_x^{4/3}}\\
\nonumber &\hspace{1.4in}+\lVert \rho_R\sigma_R\nabla|\nabla|^{s-\frac{5}{4}-\epsilon}P_{\leq N_1}F(u(t))\rVert_{L_t^2L_x^{4/3}}\\
\nonumber &\hspace{0.2in}\lesssim \left\lVert \nabla[\rho_R^{1/2}\sigma_R^{1/2}]\right\rVert_{L_t^\infty L_x^5}\lVert\rho_R^{1/2}\sigma_R^{1/2}|\nabla|^{s-\frac{5}{4}-\epsilon}P_{\leq N_1}F(u(t))\rVert_{L_t^2L_x^{20/11}}\\
&\hspace{1.4in}+\lVert \rho_R^{1/2}\sigma_R^{1/2}|\nabla|^{s-\frac{1}{4}-\epsilon}P_{\leq N_1}F(u(t))\rVert_{L_t^2L_x^{4/3}}
\end{align}

Observing the identity $F(x)=F(\tilde{\rho_R}\tilde{\sigma_R}x)+(1-\tilde{\rho_R}^3\tilde{\sigma_R}^3)F(x)$, we obtain
\begin{align}
\label{eqa1}\lVert A_1^+\rVert_{L_x^2}&\lesssim \lVert\rho_R^{1/2}\sigma_R^{1/2}|\nabla|^{s-\frac{5}{4}-\epsilon}P_{\leq N_1}F(\tilde{\rho_R}\tilde{\sigma_R}u(t))\rVert_{L_t^2L_x^{20/11}}\\
\label{eqa2}&\hspace{0.2in}+\lVert\rho_R^{1/2}\sigma_R^{1/2}|\nabla|^{s-\frac{5}{4}-\epsilon}P_{\leq N_1}(1-\tilde{\rho_R}^3\tilde{\sigma_R}^3)F(u(t))\rVert_{L_t^2L_x^{20/11}}\\
\label{eqa3}&\hspace{0.2in}+\lVert \rho_R^{1/2}\sigma_R^{1/2}\nabla|\nabla|^{s-\frac{5}{4}-\epsilon}P_{\leq N_1}F(\tilde{\rho_R}\tilde{\sigma_R}u(t))\rVert_{L_t^2L_x^{4/3}}\\
\label{eqa4}&\hspace{0.2in}+\lVert \rho_R^{1/2}\sigma_R^{1/2}\nabla|\nabla|^{s-\frac{5}{4}-\epsilon}P_{\leq N_1}(1-\tilde{\rho_R}^3\tilde{\sigma_R}^3)F(u(t))\rVert_{L_t^2L_x^{4/3}}.
\end{align}

The rest of the proof of Proposition $\ref{lem53}$ is devoted to estimating the terms ($\ref{eqa1}$)-($\ref{eqa4}$).  To estimate ($\ref{eqa1}$) and ($\ref{eqa3}$), we perform a dyadic decomposition in time.  The advantage of this decomposition arises out of the following lemma, which allows us to obtain uniform bounds on a localized Strichartz norm of the solution $u$.  
\begin{lemma}
\label{lem_a3}
There exists $R_0>0$ and $C>0$ such that 
\begin{align*}
\lVert \nabla(\tilde{\rho_R}\tilde{\sigma_R}u)\rVert_{L_t^3 (I;L_x^{3})}\leq C
\end{align*}
for every $R>R_0$ and $I\subset \mathbb{R}$ of the form $I=[-\frac{10}{\delta}R,\frac{10}{\delta}R]$ or $I=[T,2T]$ with $T\geq \frac{10}{\delta}R$.
\end{lemma}
We remark that Lemma $\ref{lem_a3}$ is obtained as in Lemma $7.2$ of \cite{KVNLW3} with an additional derivative which is accounted for by choosing the space-time norm in accordance with the Strichartz inequality.  The proof is based on the small data global theory, finite speed of propagation, and the subluminality result, Theorem $\ref{thm_sublum}$, after observing that $\nabla$ is a local operator and therefore behaves well with respect to the finite speed of propagation.  

With Lemma $\ref{lem_a3}$ in hand, we return to the task of estimating $(\ref{eqa1})$ and $(\ref{eqa3})$.  We set $T_0=0$, $T_j=\frac{10}{\delta}R2^{j-1}$, $j\geq 1$ and use Lemma $\ref{lem_cutoffs}$ (v) along with the Sobolev inequality and the decomposition 
\begin{align*}
\tilde{\rho_R}\tilde{\sigma_R}u(t)&=P_{<8N_1}(\tilde{\rho_R}\tilde{\sigma_R}u(t))+P_{\geq 8N_1}(\tilde{\rho_R}\tilde{\sigma_R}u(t))
\end{align*}
to obtain
\begin{align}
\nonumber &(\ref{eqa1})+(\ref{eqa3})\\
\nonumber &\hspace{0.2in}\lesssim \sum_{j=0}^\infty \lVert P_{<N_1}|\nabla|^{s-\epsilon-\frac{1}{4}}F(\tilde{\rho_R}\tilde{\sigma_R}u(t))\rVert_{L_t^2(I_j;L_x^{4/3})}\\
\nonumber &\hspace{0.2in}\lesssim \sum_{j=0}^\infty \bigg[\lVert |\nabla|^{s-\epsilon-\frac{1}{4}} [P_{<8N_1}(\tilde{\rho_R}\tilde{\sigma_R}u(t))]^3\rVert_{L_t^2(I_j;L_x^{4/3})}\\
\nonumber &\hspace{0.6in}+\lVert |\nabla|^{s-\epsilon-\frac{1}{4}} [P_{<8N_1}(\tilde{\rho_R}\tilde{\sigma_R}u(t))][P_{\geq 8N_1}(\tilde{\rho_R}\tilde{\sigma_R}u(t))]\tilde{\rho_R}\tilde{\sigma_R}u(t)\rVert_{L_t^2(I_j;L_x^{4/3})}\\
\nonumber &\hspace{0.6in}+\lVert P_{<N_1}|\nabla|^{s-\epsilon-\frac{1}{4}} [P_{\geq 8N_1}(\tilde{\rho_R}\tilde{\sigma_R}u(t))]^3\rVert_{L_t^2(I_j;L_x^{4/3})}\bigg]\\
\label{conclusum}&\hspace{0.2in}=:\sum_{j=0}^\infty (I)_j+(II)_j+(III)_j
\end{align}
with $I_j=[T_j,T_{j+1}]$.  In the interest of simplifying notation for the rest of the proof, all $L_t^q$ norms in the subsequent argument will be on the interval $I_j$.

To bound $(I)_j$, we consider two cases: $s>\frac{5}{4}$ and $s\leq \frac{5}{4}$.  If $s>\frac{5}{4}$ we use the H\"older inequality in time followed by the fractional product rule, interpolation, $(\ref{eqrem2})$, Proposition $\ref{prop2}$, Proposition $\ref{prop_quant}$, and the Bernstein inequalities (provided $\epsilon<\frac{7}{150}$) to obtain
\begin{align}
\nonumber (I)_j&\lesssim (2^jR)^{\frac{3}{14}}\lVert |\nabla|^{s-\epsilon-\frac{1}{4}} [P_{\leq 8N_1}(\tilde{\rho_R}\tilde{\sigma_R}u(t))]^3\rVert_{L_t^{7/2}L_x^{4/3}}\\
\nonumber &\lesssim (2^jR)^{\frac{3}{14}}\lVert P_{\leq 8N_1}(\tilde{\rho_R}\tilde{\sigma_R}u(t))\rVert_{L_t^\infty L_x^5}^{199/100}\lVert P_{\leq 8N_1}(\tilde{\rho_R}\tilde{\sigma_R}u(t))\rVert_{L_t^\infty L_x^{3}}^{1/100}\\
\nonumber &\hspace{0.2in}\cdot\lVert |\nabla|^{s-\epsilon-\frac{1}{4}} P_{\leq 8N_1}(\tilde{\rho_R}\tilde{\sigma_R}u(t))\rVert_{L_{t,x}^3}^{6/7}\lVert |\nabla|^{s-\epsilon-\frac{1}{4}} P_{\leq 8N_1}(\tilde{\rho_R}\tilde{\sigma_R}u(t))\rVert_{L_{t}^\infty L_x^{1500/661}}^{1/7}\\
\nonumber &\lesssim_u [R(1+2^j)]^{\frac{3}{14}-\frac{199}{100}(\frac{3}{8}-\frac{1}{100})}(8N_1)^{s-\epsilon-\frac{5}{4}}(8N_1)^{\frac{1}{150}-\frac{\epsilon}{7}}\\
\nonumber &\hspace{0.2in}\cdot\lVert \nabla(\tilde{\rho_R}\tilde{\sigma_R}u(t))\rVert_{L_{t,x}^3}^{6/7}\lVert |\nabla|^{s-\epsilon-\frac{1}{4}}(\tilde{\rho_R}\tilde{\sigma_R}u(t))\rVert_{L_{t}^\infty L_x^{\frac{20}{9-4\epsilon}}}^{1/7}\\
\nonumber &\lesssim_{u,N_1} [R(1+2^j)]^{-\frac{51}{100}}\lVert |\nabla|^{s} (\tilde{\rho_R}\tilde{\sigma_R}u(t))\rVert_{L_{t}^\infty L_x^{2}}^{1/7}\\
\nonumber &\lesssim_{u,N_1} [R(1+2^j)]^{-\frac{51}{100}}\lVert |\nabla|^su(t)\rVert_{L_{t}^\infty L_x^{2}}^{1/7}\\
\label{conclu1}&\lesssim_{u,N_1} [R(1+2^j)]^{-\frac{51}{100}},
\end{align}
where to obtain the fourth inequality we use $R\geq 1$, Lemma $\ref{lem_a3}$ and the Sobolev embedding.  Moreover, in obtaining the second to last inequality, we invoke the fractional product rule, Sobolev embedding and the bounds ($\ref{dbound1}$) and ($\ref{dbound2}$) to obtain the estimate
\begin{align}
\nonumber \lVert |\nabla|^{s} (\tilde{\rho_R}\tilde{\sigma_R}u(t))\rVert_{L_{t}^\infty L_x^{2}}&\lesssim \lVert |\nabla|^{s} (\tilde{\rho_R}\tilde{\sigma_R})\rVert_{L_t^\infty L_x^{\frac{5}{s}}}\lVert u(t)\rVert_{L_{t}^\infty L_x^{\frac{10}{5-2s}}}\\
\nonumber &\hspace{0.2in}+\lVert (\tilde{\rho_R}\tilde{\sigma_R})\rVert_{L_{t,x}^\infty}\lVert |\nabla|^su(t)\rVert_{L_{t}^\infty L_x^{2}}\\
\label{eqchiu}&\lesssim \lVert |\nabla|^su(t)\rVert_{L_{t}^\infty L_x^{2}}.
\end{align}

It now remains to consider the case $s\leq \frac{5}{4}$.  Towards this end, fixing $w\in (\frac{10(3+2s)}{15+4s},5]$, we proceed in a similar manner as above: using H\"older in time, followed by the fractional product rule, interpolation, Proposition $\ref{prop2}$, Proposition $\ref{prop_quant}$, ($\ref{eqrem2}$), and the Bernstein and Sobolev inequalities, we obtain
\begin{align}
\nonumber (I)_j&\lesssim (2^jR)^{\frac{3(3-2s)}{2(11-6s)}}\lVert |\nabla|^{s-\epsilon-\frac{1}{4}} [P_{\leq 8N_1}(\tilde{\rho_R}\tilde{\sigma_R}u(t))]^3\rVert_{L_t^{11-6s}L_x^{4/3}}\\
\nonumber &\lesssim (2^jR)^{\frac{3(3-2s)}{2(11-6s)}}\lVert P_{\leq 8N_1}(\tilde{\rho_R}\tilde{\sigma_R}u(t))\rVert_{L_t^\infty L_x^5}^{199/100}\lVert P_{\leq 8N_1}(\tilde{\rho_R}\tilde{\sigma_R}u(t))\rVert_{L_t^\infty L_x^{3}}^{1/100}\\
\nonumber &\hspace{0.2in}\cdot\lVert |\nabla|^{s-\epsilon-\frac{1}{4}} P_{\leq 8N_1}(\tilde{\rho_R}\tilde{\sigma_R}u(t))\rVert_{L_{t}^3L_x^{\sigma_1}}^{\frac{3}{11-6s}}\lVert |\nabla|^{s-\epsilon-\frac{1}{4}} P_{\leq 8N_1}(\tilde{\rho_R}\tilde{\sigma_R}u(t))\rVert_{L_{t}^\infty L_x^{\sigma_2}}^{\frac{2(4-3s)}{11-6s}}\\
\nonumber &\lesssim_u (2^jR)^{\frac{3(3-2s)}{2(11-6s)}+\frac{199}{100}(-\frac{8(5-w)}{5(12-w)})}(8N_1)^{\frac{2(4-3s)}{11-6s}(\frac{9}{4}-\epsilon-\frac{5}{\sigma_2})}\\
\nonumber &\hspace{0.2in}\cdot\lVert |\nabla|^{s-\epsilon-\frac{1}{4}} P_{\leq 8N_1}(\tilde{\rho_R}\tilde{\sigma_R}u(t))\rVert_{L_{t}^3L_x^{\sigma_1}}^{\frac{3}{11-6s}}\lVert |\nabla|^{s-\epsilon-\frac{1}{4}} P_{\leq 8N_1}(\tilde{\rho_R}\tilde{\sigma_R}u(t))\rVert_{L_{t}^\infty L_x^{\frac{20}{9-4\epsilon}}}^{\frac{2(4-3s)}{11-6s}}\\
\nonumber &\lesssim_{u,N_1} (2^jR)^{\frac{3(3-2s)}{2(11-6s)}+\frac{199}{100}(-\frac{8(5-w)}{5(12-w)})}\\
\nonumber &\hspace{0.2in}\cdot\lVert \nabla(\tilde{\rho_R}\tilde{\sigma_R}u(t))\rVert_{L_{t,x}^3}^{\frac{3}{11-6s}}\lVert |\nabla|^{s} (\tilde{\rho_R}\tilde{\sigma_R}u(t))\rVert_{L_{t}^\infty L_x^{2}}^{\frac{2(4-3s)}{11-6s}}\\
\nonumber &\lesssim_{u,N_1} (2^jR)^{\frac{3(3-2s)}{2(11-6s)}+\frac{199}{100}(-\frac{8(5-w)}{5(12-w)})}\lVert |\nabla|^su(t))\rVert_{L_{t}^\infty L_x^{2}}^{\frac{2(4-3s)}{11-6s}}\\
\label{conclu2}&\lesssim_{u,N_1} [R(1+2^j)]^{-\frac{501}{1000}}.
\end{align}
for $\epsilon<\frac{1}{150}$, where we have set
\begin{align*}
\sigma_1=\tfrac{60}{5+12s-12\epsilon}\quad\textrm{and}\quad\sigma_2=\tfrac{1500(4-3s)}{2689-2019s+450\epsilon},
\end{align*}
chosen $w$ sufficiently close to $\frac{10(3+2s)}{15+4s}$, and recalled that we have $R\geq 1$.  We remark that in order to apply Bernstein in the third inequality, we have observed $\sigma_2>\frac{20}{9-4\epsilon}$ when $\epsilon<\frac{1}{150}$ and $s\leq \frac{5}{4}$.  We also note that to obtain the fifth inequality, we used Lemma $\ref{lem_a3}$ and ($\ref{eqchiu}$).

It now remains to estimate the terms $(II)_j$ and $(III)_j$.  For this, we use H\"older in time and space, followed by interpolation and the Bernstein inequalities to obtain
\begin{align}
\nonumber &(II)_j+(III)_j\\
\nonumber &\hspace{0.4in}\lesssim N_1^{s-\epsilon-\frac{1}{4}}(2^jR)^{\frac{1}{5}}\lVert \tilde{\rho_R}\tilde{\sigma_R}u(t)\rVert_{L_t^\infty L_x^5}^2\lVert P_{\geq 8N_1}(\tilde{\rho_R}\tilde{\sigma_R}u(t))\rVert_{L_t^{10/3}L_x^{20/7}}\\
\nonumber &\hspace{0.4in}\lesssim_u N_1^{s-\epsilon-\frac{1}{4}}(2^jR)^{\frac{1}{5}}\lVert \tilde{\rho_R}\tilde{\sigma_R}u(t)\rVert_{L_t^\infty L_x^5}^{2}\\
\nonumber &\hspace{0.6in}\cdot \lVert P_{\geq 8N_1}(\tilde{\rho_R}\tilde{\sigma_R}u(t))\rVert_{L_{t,x}^3}^{9/10}\lVert P_{\geq 8N_1}(\tilde{\rho_R}\tilde{\sigma_R}u(t))\rVert_{L_t^\infty L_x^{2}}^{1/10}\\
\nonumber &\hspace{0.4in}\lesssim_u N_1^{\frac{9s}{10}-\frac{23}{20}-\epsilon}[R(1+2^j)]^{-\frac{23}{45}}\lVert \nabla (\tilde{\rho_R}\tilde{\sigma_R}u(t))\rVert_{L_{t,x}^3}^{9/10}\lVert |\nabla|^s(\tilde{\rho_R}\tilde{\sigma_R}u(t))\rVert_{L_t^\infty L_x^{2}}^{1/10}\\
&\hspace{0.4in}\lesssim_{u,N_1} [R(1+2^j)]^{-\frac{23}{45}},\label{twoandthree}
\end{align}
where we have used $(\ref{eqrem2})$ and Proposition $\ref{prop_quant}$ with $w=3$ to obtain the second to last inequality, and Lemma $\ref{lem_a3}$ and ($\ref{eqchiu}$) to obtain the last inequality.

Combining $(\ref{conclu1})$-$(\ref{conclu2})$ and $(\ref{twoandthree})$ with ($\ref{conclusum}$), we obtain
\begin{align}
\label{conclu1a}(\ref{eqa1})+(\ref{eqa3})&\lesssim_{u,N_1} R^{-\frac{501}{1000}}
\end{align}
whenever $\epsilon$ is sufficiently small.

We now estimate ($\ref{eqa2}$) and ($\ref{eqa4}$).  Note that by the Sobolev inequality, we have
\begin{align}
(\ref{eqa2})+(\ref{eqa4})&\lesssim \lVert \rho_R^{1/2}\sigma_R^{1/2}|\nabla|^{s-\epsilon-\frac{1}{4}}P_{\leq N_1}(1-\tilde{\rho}_R^3\tilde{\sigma}_R^3)F(u(t))\rVert_{L_t^2L_x^{4/3}}\\
&\lesssim \lVert (|t|+R)^{-s+\epsilon-1}(1-\tilde{\rho}_R^3\tilde{\sigma}_R^3)F(u(t))\rVert_{L_t^2L_x^1}\\
\nonumber &\lesssim \lVert (|t|+R)^{-s+\epsilon-1}\rVert_{L_t^2}\lVert F(u(t))\rVert_{L_t^\infty L_x^1}\\
\nonumber &\lesssim_u R^{-\frac{1}{2}-(s-\epsilon)}\\
\label{conclu3}&\lesssim_u R^{-\frac{1}{2}-\beta}
\end{align}
whenever $\beta<s-\epsilon$, where we have invoked Lemma $\ref{lem_mismatch}$ with $\sigma=s-\epsilon-\frac{1}{4}$, $A=C_1(|t|+R)$, $q=\frac{4}{3}$ and $p=1$.  The claim follows by combining $(\ref{conclu1a})$ and $\ref{conclu3}$ and observing that for $\epsilon<\frac{1}{2}$, $\beta<\frac{1}{1000}$ implies $\beta<\frac{1}{2}<s-\epsilon$.  This completes the proof of Proposition $\ref{lem53}$.
\end{proof}

\begin{proposition}
\label{lem54}
We have the bounds
\begin{align*}
\max \{\lVert A_2^+\rVert_{L_x^2},\lVert A_2^-\rVert_{L_x^2},\lVert B_2^+\rVert_{L_x^2},\lVert B_2^-\rVert_{L_x^2}\}\lesssim R^{\frac{1}{2}}.
\end{align*}
\end{proposition}

\begin{proof}
The proof proceeds as in the proof of Lemma $7.4$ of \cite{KVNLW3}.  We show the estimate for $\lVert A_2^+\rVert_{L_x^2}$ and remark that the other estimates are similar.  Setting $I_R=[\frac{R}{2},\frac{10R}{\delta}]$, we use the inhomogeneous Strichartz inequality followed by the fractional product rule, H\"older's inequality, and the Sobolev and Bernstein inequalities to obtain
\begin{align}
\nonumber \lVert A_2^+\rVert_{L_x^2}&\leq \lim_{T\rightarrow\infty} \bigg\lVert \int_0^T \nabla\frac{\sin(-t|\nabla|)}{|\nabla|}\rho_R(t,x)(1-\sigma_R(t,x))|\nabla|^{s-1-\epsilon}P_{\leq N_1}F(u(t))dt\bigg\rVert_{L_x^2}\\
\nonumber&\lesssim \lVert \nabla [\rho_R(t,x)(1-\sigma_R(t,x))|\nabla|^{s-1-\epsilon}P_{\leq N_1}F(u(t))]\rVert_{L_t^{2}L_x^{5/4}(I_R\times\mathbb{R}^5)}\\
\nonumber&\lesssim \lVert \nabla[\rho_R(t,x)(1-\sigma_R(t,x))]\rVert_{L_t^\infty L_x^5(I_R\times\mathbb{R}^5)}\lVert |\nabla|^{s-1-\epsilon}P_{\leq N_1}F(u(t))\rVert_{L_t^2L_x^{5/3}(I_R\times \mathbb{R}^5)}\\
\nonumber &\hspace{0.2in}+\lVert \rho_R(1-\sigma_R)\rVert_{L_t^\infty L_x^\infty}\lVert |\nabla|^{s-\epsilon}P_{\leq N_1}F(u(t))\rVert_{L_t^2L_x^{5/4}(I_R\times \mathbb{R}^5)}\\
\nonumber &\lesssim \lVert |\nabla|^{s-\epsilon}P_{\leq N_1}F(u(t))\rVert_{L_t^2L_x^{5/4}(I_R\times \mathbb{R}^5)}\\
\nonumber &\lesssim \lVert F(u(t))\rVert_{L_t^\infty L_x^{5/4}}|I_R|^\frac{1}{2}\\
\nonumber &\lesssim\lVert u\rVert_{L_t^\infty L_x^{15/4}}^3R^\frac{1}{2}\\
\nonumber&\lesssim R^\frac{1}{2}.
\end{align}
where we have used $\ref{eqrem2}$ interpolated with the a priori bound $(u,u_t)\in L_t^\infty(\dot{H}_x^{3/2}\times\dot{H}_x^{1/2})$ to obtain the last inequality. 
\end{proof}

To estimate the last term in ($\ref{ceq3}$), we will make use of the following weak diffraction lemma, the proof of which we give in Appendix A.  This result is the analogue of the weak diffraction property presented in \cite{KVNLW3}, adapted to the dimension $d=5$ setting.
\begin{lemma} (Weak diffraction)\label{thm_weakdiff}
Fix $\phi\in C^\infty(\mathbb{R}^5;[0,1])$ such that $\phi(x)=1$ for $|x|<1$ and $\phi(x)=0$ for $|x|>2$.  Then there exists $C>0$ such that if $F,G:\mathbb{R}\times\mathbb{R}^5\rightarrow\mathbb{R}$ are given along with $R_0,C_1,C_2>0$ such that for every $R\geq R_0$ we have
\begin{align*}
&\supp F\times \supp G\subset \\
&\hspace{.5in}\bigg\{\big((t,x),(\tau,y)\big):|t|+|\tau|+|x|+|y|\leq C_1R, |t-\tau|-|x-y|\geq C_2R\bigg\},
\end{align*}
then for every $R\geq R_0$ and $x_0\in\mathbb{R}^5$ we have
\begin{align}
|I(F,G,R)|&\leq CR^{-1/26}\lVert F\rVert_{L_t^\infty L_x^1}\lVert G\rVert_{L_t^\infty L_x^1}\label{goal1}
\end{align}
where
\begin{align*}
I(F,G,R)&:=\int_{\mathbb{R}}\int_{\mathbb{R}} \langle \nabla\frac{\sin(t|\nabla|)}{|\nabla|}\theta(i\nabla)F(t),\phi(\frac{\cdot-x_0}{R})\nabla\frac{\sin(\tau|\nabla|)}{|\nabla|}\theta(i\nabla)G(\tau)\rangle d\tau dt\\
\nonumber &\hspace{0.2in}+\int_{\mathbb{R}}\int_{\mathbb{R}} \langle \cos(t|\nabla|)\theta(i\nabla)F(t),\phi(\frac{\cdot-x_0}{R})\cos(\tau|\nabla|)\theta(i\nabla)G(\tau)\rangle d\tau dt.
\end{align*}
where $\theta$ is defined in ($\ref{eq_theta}$).
\end{lemma}

\begin{proposition}
We have 
\begin{align*}
|\langle A_2^+,\phi_{R}A_2^-\rangle+\langle B_2^+,\phi_{R}B_2^-\rangle|\lesssim R^{-\beta}
\end{align*}
for every $\beta<\frac{1}{26}$.
\label{lem55}
\end{proposition}

\begin{proof}
The proof proceeds as in the proof of Lemma $7.5$ of \cite{KVNLW3}.  We apply Lemma $\ref{thm_weakdiff}$ with 
\begin{align*}
\tilde{F}(t)=\rho_R(t,x)(1-\sigma_R(t,x))|\nabla|^{s-1-\epsilon}P_{\leq N_1}F(u(t)),\\ G(\tau)=\rho_R(\tau,x)(1-\sigma_R(\tau,x))|\nabla|^{s-1-\epsilon}P_{\leq N_1}F(u(\tau)).
\end{align*}
Note that the hypotheses of the lemma imply that $F$ and $G$ have the required support.  We therefore conclude
\begin{align*}
|\langle A_2^+,\phi_{R}A_2^-\rangle+\langle B_2^+,\phi_{R}B_2^-\rangle|
&\lesssim R^{-1/26}\lVert \tilde{F}\rVert_{L_t^\infty L_x^1}\lVert G\rVert_{L_t^\infty L_x^1}\\
&\lesssim R^{-1/26}\lVert F(u)\rVert_{L_t^\infty L_x^1}^2\\
&\lesssim R^{-1/26},
\end{align*}
where we have used the Bernstein inequality along with the condition $s-1-\epsilon\geq 0$ to obtain the second inequality, and Proposition $\ref{prop2}$ to obtain the third inequality.  Then $R\geq 1$ implies $R^{-1/26}\leq R^{-\beta}$ for every $\beta<\frac{1}{26}$, which gives the desired result, completing the proof of Proposition $\ref{lem55}$.
\end{proof}

We are now ready to complete the proof of Lemma $\ref{lem_somenegreg}$.  In particular, collecting Proposition $\ref{lem53}$, Proposition $\ref{lem54}$ and Proposition $\ref{lem55}$, and invoking ($\ref{ceq3}$) we obtain
\begin{align*}
&\lVert P_{\leq N_1}|\nabla|^{s-\epsilon}u(0)\rVert_{L_x^2(B_R)}+\lVert P_{\leq N_1}|\nabla|^{s-1-\epsilon} u_t(0)\rVert_{L_x^2(B_R)}\\
&\hspace{0.2in}\lesssim 2\bigg(R^{-2(\frac{1}{2}+\beta)}+2R^{-\frac{1}{2}-\beta}R^{\frac{1}{2}}+R^{-\beta}\bigg)\\
&\hspace{0.2in}\lesssim R^{-\beta/2}.
\end{align*}
whenever $\epsilon$ and $\beta$ are sufficiently small, which concludes the proof of Lemma $\ref{lem_somenegreg}$.
\end{proof}

\begin{proof}[Proof of Theorem $\ref{thm_negreg}$.]
Iteratively apply Lemma $\ref{lem_somenegreg}$, starting with $s=3/2$, to obtain $(u,u_t)\in L_t^\infty(\dot{H}_x^{s-\epsilon}\times\dot{H}_x^{s-\epsilon-1})$, with $\epsilon$ sufficiently small to satisfy the hypotheses of Lemma $\ref{lem_somenegreg}$.  The claim then follows after finitely many iterations.
\end{proof}

\section{Proof of Theorem $\ref{thm1}$.}

In this section we conclude the proof of Theorem $\ref{thm1}$ by precluding each of the scenarios identified in Theorem $\ref{thm_apred}$.  We begin with the finite time blow-up solution.  As we mentioned in the introduction, the argument used to rule out this scenario in high dimensions $d\geq 6$ in \cite{BulutCubic} is also applicable to the present case.  In particular, we outline the arguments used in the proof below.
\begin{proposition}[Finite time blow-up solution]
\label{prop_fintime}
There is no solution $u:I\times\mathbb{R}^5\rightarrow\mathbb{R}$ with maximal interval of existence $I$ satisfying the conditions of a finite time blow-up solution as in Theorem $\ref{thm_apred}$.
\end{proposition}
\begin{proof}
We argue as in \cite{BulutCubic}; see also \cite{BulutRadial,KMESupercriticalNLW,KVNLW3,KVNLWradial}.  Without loss of generality (using the time translation and time reversal symmetries), suppose that $\sup I=1$.  The first step is to show that there exists $x_0\in \mathbb{R}^5$ such that
\begin{align}
\supp u(t),\quad \supp u_t(t)\subset \overline{B(x_0,1-t)}\label{support}
\end{align}
for all $t\in I$.  This is accomplished by making use of the finite speed of propagation along with the definition of almost periodicity (for details, we refer the reader to the proof of Lemma $6.2$ in \cite{BulutCubic}, and the references cited therein).

We now estimate the energy, using ($\ref{support}$), H\"older's inequality and the Sobolev embedding to obtain
\begin{align*}
E(u(t),u_t(t))&=\int_{|x-x_0|\leq 1-t} \frac{1}{2}|\nabla u(t)|^2+\frac{1}{2}|u_t(t)|^2+\frac{1}{4}|u(t)|^4dx\\
&\lesssim (1-t)[\lVert \nabla u\rVert_{L_t^\infty L_x^{5/2}}^2+\lVert u_t\rVert_{L_t^\infty L_x^{5/2}}^2+\lVert u\rVert_{L_t^\infty L_x^5}^4]\\
&\lesssim (1-t)\lVert (u,u_t)\rVert_{L_t^\infty(\dot{H}_x^{3/2}\times\dot{H}_x^{1/2})}^2\\
&\lesssim (1-t)
\end{align*}
where we have used the a priori bound $(u,u_t)\in L_t^\infty(\dot{H}_x^{3/2}\times\dot{H}_x^{1/2})$.  Letting $t\rightarrow 1$ and using the conservation of energy, we obtain
\begin{align*}
E(u(0),u_t(0))=0,
\end{align*}
which gives $u\equiv 0$, contradicting $\lVert u\rVert_{L_{t,x}^6}=\infty$.
\end{proof}

We now turn to the two global scenarios identified in Theorem $\ref{thm_apred}$: the soliton-like solution and the low-to-high frequency cascade solution.  We remark that the essential ingredient in precluding these scenarios is Theorem $\ref{thm_negreg}$.  We begin with the soliton-like solution.
\begin{proposition}[Soliton-like solution]
\label{prop_sol}
There is no solution $u:\mathbb{R}\times\mathbb{R}^5\rightarrow\mathbb{R}$ satisfying the conditions of a soliton-like solution as in Theorem $\ref{thm_apred}$.
\end{proposition}
\begin{proof}
The proof proceeds as in \cite{KVNLW3}; see also Section $8$ of \cite{BulutCubic}.  Suppose for contradiction that such a solution $u$ exists.  Let $T>0$ be given.  Invoking \cite[Lemma $2.6$]{KVNLWradial} (see also Lemma $8.3$ of \cite{BulutCubic}), we obtain
\begin{align*}
\int_0^T \int_{\mathbb{R}^5} \frac{|u(t,x)|^4}{|x|}dxdt&\geq \sum_{i=0}^{\lfloor T\rfloor-1}\int_i^{i+1}\int_{|x-x(t)|\leq R} \frac{|u(t,x)|^4}{|x|}dxdt\\
&\geq \sum_{i=0}^{\lfloor T\rfloor-1}\frac{1}{C'+i}\int_i^{i+1}\int_{|x-x(t)|\leq R} |u(t,x)|^4dxdt\\
&\geq c\log\left(\frac{C'+\lfloor T\rfloor}{C'}\right).
\end{align*}
On the other hand, the Morawetz estimate
\begin{align*}
\int_0^T\int_{\mathbb{R}^5} \frac{|u(t,x)|^4}{|x|}dxdt\lesssim E(u_0,u_1)
\end{align*}
gives
\begin{align*}
&\log\left(\frac{C'+\lfloor T\rfloor}{C'}\right)\lesssim \lVert (u,u_t)\rVert_{L_t^\infty(\dot{H}_x^{1}\times L_x^2)}^2+\lVert u\rVert_{L_t^\infty L_x^4}^4\\
&\hspace{0.2in}\lesssim \lVert (u,u_t)\rVert_{L_t^\infty(\dot{H}_x^{1}\times L_x^2)}^2+\lVert u\rVert_{L_t^\infty L_x^5}^{5/8}\lVert u\rVert_{L_t^\infty L_x^3}^{3/8}\\
&\hspace{0.2in}\lesssim 1,
\end{align*}
where we have used Theorem $\ref{thm_negreg}$ and ($\ref{eqrem2}$) to see the finiteness of the right hand side.  Taking $T\rightarrow\infty$, we obtain a contradiction as desired.
\end{proof}

To conclude, it therefore suffices to rule out the low-to-high frequency cascade solution of Theorem $\ref{thm_apred}$.  For this, we first recall that for any almost periodic solution $u$ to (NLW), the compactness characterization ($\ref{cptness}$) of almost periodicity implies that there exists $c(\eta)>0$ such that
\begin{align}
\int_{|\xi|\leq c(\eta)N(t)} ||\xi|^{3/2}\hat{u}(t,\xi)|^2+||\xi|^{1/2}\hat{u}_t(t,\xi)|^2d\xi<\eta\label{ap_2}
\end{align}
(see for instance, Remark $3.4$ of \cite{BulutCubic} and $(8.1)$ in \cite{KVNLW3}).  We then have
\begin{proposition}[Low-to-high frequency cascade solution]
\label{prop_casc}
There is no solution $u:\mathbb{R}\times\mathbb{R}^5\rightarrow\mathbb{R}$ satisfying the conditions of a low-to-high frequency cascade solution as in Theorem $\ref{thm_apred}$.
\end{proposition}
\begin{proof}
We proceed as in \cite{KVNLW3}; see also Section $9$ of \cite{BulutCubic}.  Assuming for contradiction that such a $u$ existed, we choose a sequence $\{t_n\}$ such that $t_n\rightarrow\infty$ and $N(t_n)\rightarrow\infty$ as $n\rightarrow\infty$.  Fix $\eta>0$, choose $c(\eta)$ as in ($\ref{ap_2}$), and fix a dyadic number $M\in (0,\frac{1}{2})$.  Suppose that $n$ is large enough to ensure $M<c(\eta)N(t_n)$.  We may then write
\begin{align*}
u_{\leq c(\eta)N(t_n)}(t_n)=u_{\leq M}(t_n)+u_{M<\cdot\leq c(\eta)N(t_n)}(t_n),
\end{align*}
so that by applying the Bernstein inequalities, H\"older's inequality, and ($\ref{ap_2}$), we obtain
\begin{align}
\nonumber \lVert \nabla u_{\leq c(\eta)N(t_n)}(t_n)\rVert_{L_x^2}&\leq \lVert \nabla u_{\leq M}(t_n)\rVert_{L_x^2}+\lVert \nabla u_{M<\cdot \leq c(\eta)N(t_n)}\rVert_{L_x^2}\\
\nonumber &\lesssim M^{\frac{5\beta}{6}}\lVert \nabla u_{\leq M}(t_n)\rVert_{L_x^\frac{6}{3+\beta}}+M^{-1/2}\lVert |\nabla|^{3/2} u_{\leq c(\eta)N(t_n)}(t_n)\rVert_{L_x^2}\\
\nonumber &\lesssim M^{\frac{5\beta}{6}}\lVert \langle x-x(t)\rangle^\beta\nabla u_{\leq 1}\rVert_{L_t^\infty L_x^2}\lVert \langle x-x(t)\rangle^{-\beta}\rVert_{L_t^\infty L_x^{6/\beta}}+M^{-1/2}\eta\\
&\lesssim M^{\frac{5\beta}{6}}+M^{-1/2}\eta,\label{est-11}
\end{align}
where we have invoked Theorem $\ref{thm_negreg}$ and noted that
\begin{align*}
\lVert \langle x-x(t)\rangle^{-\beta}\rVert_{L_x^{6/\beta}}=\left(\int_{\mathbb{R}^5} \langle x-x(t)\rangle^{-6}dx\right)^{\beta/6}<\infty
\end{align*}
for any $t\in\mathbb{R}$ to obtain the last inequality.

On the other hand, the Bernstein inequalities may also be applied to the high frequency portion of $u(t_n)$, giving
\begin{align}
\nonumber \lVert \nabla u_{>c(\eta)N(t_n)}(t_n)\rVert_{L_x^2}&\lesssim [c(\eta)N(t_n)]^{-1/2}\lVert |\nabla|^{3/2}u_{>c(\eta)N(t_n)}(t_n)\rVert_{L_x^2}\\
\nonumber &\lesssim [c(\eta)N(t_n)]^{-1/2}\lVert |\nabla|^{3/2}u\rVert_{L_t^\infty L_x^2}\\
&\lesssim [c(\eta)N(t_n)]^{-1/2}\label{est-12}
\end{align}
where we have used the a priori bound $(u,u_t)\in L_t^\infty(\dot{H}_x^{3/2}\times\dot{H}_x^{1/2})$.  Combining ($\ref{est-11}$) and ($\ref{est-12}$) then gives the bound
\begin{align}
\lVert \nabla u(t_n)\rVert_{L_x^2}&\leq M^\frac{5\beta}{6}+M^{-1/2}\eta+[c(\eta)N(t_n)]^{-1/2}.\label{est-21}
\end{align}

Repeating these arguments for $u_t(t_n)$, we obtain
\begin{align*}
\lVert P_{\leq c(\eta)N(t_n)}u_t(t_n)\rVert_{L_x^2}&\lesssim M^{\frac{5\beta}{6}}+M^{-1/2}\eta
\end{align*}
and
\begin{align*}
\lVert P_{>c(\eta)N(t_n)}u_t(t_n)\rVert_{L_x^2}&\lesssim [c(\eta)N(t_n)]^{-1/2},
\end{align*}
and we therefore obtain
\begin{align}
\lVert u_t(t_n)\rVert_{L_x^2}&\leq M^\frac{5\beta}{6}+M^{-1/2}\eta+[c(\eta)N(t_n)]^{-1/2}.\label{est-22}
\end{align}

We now estimate the potential energy.  For this, we note that the Sobolev embedding followed by interpolation gives
\begin{align}
\nonumber \lVert u(t_n)\rVert_{L_x^4}&\leq \lVert |\nabla|^{5/4}u(t_n)\rVert_{L_x^2}\\
\nonumber &\lesssim \lVert \nabla u(t_n)\rVert_{L_x^2}^{1/2}\lVert |\nabla|^{3/2}u(t_n)\rVert_{L_x^2}^{1/2}\\
&\lesssim \left(M^\frac{5\beta}{6}+M^{-1/2}\eta+[c(\eta)N(t_n)]^{-1/2}\right)^{1/2},\label{est-23}
\end{align}
where we have again used the a priori bound $(u,u_t)\in L_t^\infty(\dot{H}_x^{3/2}\times\dot{H}_x^{1/2})$.

Combining ($\ref{est-21}$), ($\ref{est-22}$) and ($\ref{est-23}$), we obtain
\begin{align}
E(u_0,u_1)&=E(u(t_n),u_t(t_n))\lesssim \omega+\omega^{1/2},\label{est-31}
\end{align}
where
\begin{align*}
\omega=M^\frac{5\beta}{6}+M^{-1/2}\eta+[c(\eta)N(t_n)]^{-1/2}.
\end{align*}
Recall that ($\ref{est-31}$) holds for all $n$ sufficiently large (depending on $M$ and $\eta$).  Fixing $M$ and $\eta$, we now let $n\rightarrow\infty$ to obtain
\begin{align*}
E(u_0,u_1)&\lesssim M^\frac{5\beta}{6}+M^{-1/2}\eta+\left(M^\frac{5\beta}{6}+M^{-1/2}\eta\right)^{1/2}.
\end{align*}
We now let $\eta\rightarrow 0$ followed by $M\rightarrow 0$ to conclude $E(u_0,u_1)=0$ and thus $u\equiv 0$, contradicting $\lVert u\rVert_{L_{t,x}^6}=\infty$.
\end{proof}

\appendix
\section{Proof of Lemma $\ref{thm_weakdiff}$.}

In this appendix, we prove the weak diffraction result, Lemma $\ref{thm_weakdiff}$, which is used in the proof of Proposition $\ref{lem55}$.

\begin{proof}
We argue as in Proposition $2.6$ of \cite{KVNLW3}.  To prove the lemma, we decompose the integral $I(F,G,R)$ into the sum of five terms, each of which will be estimated individually.  Toward this end, we use the Plancherel theorem and the identity $\widehat{\phi(\frac{\cdot-x_0}{R})}(\xi)=R^5e^{-ix_0\cdot \xi}\hat{\phi}(R\xi)$ to write
\begin{align*}
&\int \langle \nabla \frac{\sin(t|\nabla|)}{|\nabla|}\theta(i\nabla)F(t), \,\,  \phi(\frac{\cdot-x_0}{R})\nabla\frac{\sin(\tau|\nabla|)}{|\nabla|}\theta(i\nabla)G(\tau)\rangle d\tau dt\\
&\hspace{0.2in}=\int \langle (i\xi)\frac{\sin(t|\xi|)}{|\xi|}\theta(\xi)\hat{F}(t,\xi), \,\, \widehat{\phi\left(\frac{\cdot-x_0}{R}\right)}*(i\,\cdot\,)\frac{\sin(\tau |\cdot|)}{|\cdot|}\theta(\cdot)\hat{G}(\tau)\rangle d\tau dt\\
&\hspace{0.2in}=\int R^5\sin(t|\xi|)\sin(\tau|\eta|)\theta(\xi)\theta(\eta)\frac{\xi\cdot \eta}{|\xi||\eta|}\\
&\hspace{1.6in}\,\cdot\,\overline{e^{-ix_0\cdot R(\xi-\eta)}\widehat{\phi}(R(\xi-\eta))}\hat{F}(t,\xi)\overline{\hat{G}(t,\eta)}d\eta d\xi d\tau dt.
\end{align*}
Summing this with a corresponding identity for the operator $\cos(t|\nabla|)$ and using the changes of variables $\xi\mapsto 2\xi+\eta$, $\eta\mapsto \eta-\xi$, we find
\begin{align*}
I(F,G,R)&=\int R^5\bigg(\cos(t|\xi|-\tau |\eta|)-\sin(t|\xi|)\sin(\tau|\eta|)(1-\frac{\xi\cdot \eta}{|\xi||\eta|})\bigg)\\
&\hspace{0.8in}\cdot\theta(\xi)\theta(\eta)\overline{e^{-ix_0\cdot R(\xi-\eta)}\widehat{\phi}(R(\xi-\eta))}\hat{F}(t,\xi)\overline{\hat{G}(\tau,\eta)}d\eta d\xi d\tau dt\\
&=\int \frac{R^5}{2}\bigg(\textrm{Re}[e^{i(t|\xi|-\tau|\eta|)}](1+\frac{\xi\cdot \eta}{|\xi||\eta|})+\textrm{Re}[e^{i(t|\xi|+\tau |\eta|)}](1-\frac{\xi\cdot\eta}{|\xi||\eta|})\bigg)\\
&\hspace{0.8in}\cdot\theta(\xi)\theta(\eta)\overline{e^{-ix_0\cdot R(\xi-\eta)}\widehat{\phi}(R(\xi-\eta))}\hat{F}(t,\xi)\overline{\hat{G}(\tau,\eta)}d\eta d\xi d\tau dt\\
&=(I)+(II)+(III)+(IV)+(V),
\end{align*}
where we have set
\begin{align*}
(I)&=\int 2^4R^5A(t,\tau,\mu,\nu)\phi(R^{11/25}\mu)\hat{F}(t,\mu+\nu)\overline{\hat{G}(\tau,\mu-\nu)}d\mu d\nu d\tau dt,\\
(II)&=\int 2^4R^5A_1(t,\tau,\mu,\nu)(1-\phi(R^{11/25}\mu))\phi(R^{12/25}\nu)\hat{F}(t,\mu+\nu)\overline{\hat{G}(\tau,\mu-\nu)}d\mu d\nu d\tau dt,\\
(III)&=\int_{\{(t,\tau,\mu,\nu):|t+\tau|\leq R^{25/26}\}} 2^4R^5A_2(t,\tau,\mu,\nu)(1-\phi(R^{11/25}\mu))\phi(R^{12/25}\nu)\\
&\hspace{1in}\hat{F}(t,\mu+\nu)\overline{\hat{G}(\tau,\mu-\nu)}d\mu d\nu d\tau dt,\\
(IV)&=\int_{\{(t,\tau,\mu,\nu):|t+\tau|>R^{25/26}\}} 2^4R^5A_2(t,\tau,\mu,\nu)(1-\phi(R^{11/25}\mu))\phi(R^{12/25}\nu)\\
&\hspace{1in}\hat{F}(t,\mu+\nu)\overline{\hat{G}(\tau,\mu-\nu)}d\mu d\nu d\tau dt,\\
(V)&=\int 2^4R^5A(t,\tau,\mu,\nu)(1-\phi(R^{11/25}\mu))(1-\phi(R^{12/25}\nu))\\
&\hspace{1in}\hat{F}(t,\mu+\nu)\overline{\hat{G}(\tau,\mu-\nu)}d\mu d\nu d\tau dt,
\end{align*}
and
\begin{align*}
A(t,\tau,\mu,\nu)&:=A_1(t,\tau,\mu,\nu)+A_2(t,\tau,\mu,\nu),\\
A_1(t,\tau,\mu,\nu)&:=\textrm{Re}[e^{i(t|\mu+\nu|-\tau|\mu-\nu|)}](1+\frac{(\mu+\nu)\cdot (\mu-\nu)}{|\mu+\nu||\mu-\nu|})\\
&\hspace{0.4in}\theta(\mu+\nu)\theta(\mu-\nu)\overline{e^{-ix_0\cdot 2R\nu}\widehat{\phi}(R(2\nu))},\\
A_2(t,\tau,\mu,\nu)&:=\textrm{Re}[e^{i(t|\mu+\nu|+\tau |\mu-\nu|)}](1-\frac{(\mu+\nu)\cdot(\mu-\nu)}{|\mu+\nu||\mu-\nu|})\\
&\hspace{0.4in}\theta(\mu+\nu)\theta(\mu-\nu)\overline{e^{-ix_0\cdot 2R\nu}\widehat{\phi}(R(2\nu))}.
\end{align*}

It now remains to estimate the terms $(I)$-$(V)$.  For notational convenience, we set \begin{align*}
X_1&:=\{t:\supp F\cap (\{t\}\times \mathbb{R}^5)\neq \emptyset\}\times \{\tau:\supp G\cap (\{\tau\}\times \mathbb{R}^5)\neq \emptyset\},
\end{align*}
\begin{align*}
X&:=X_1\times \mathbb{R}^5\times\mathbb{R}^5,
\end{align*}
and note that the hypotheses of the lemma imply $|X_1|\leq (2C_1R)^2$.

To estimate $(I)$, we use the change of variables $\mu\mapsto R^{-11/25}\mu$ and $\nu\mapsto R^{-1}\nu$ to obtain
\begin{align*}
|(I)|&\leq \int_X R^5 \lVert F\rVert_{L_t^\infty L_x^1}\lVert G\rVert_{L_t^\infty L_x^1}\, |A(t,\tau,\mu,\nu)|\,|\phi(R^{11/25}\mu)| d\mu d\nu d\tau dt\\
&\leq \int_X R^5 \lVert F\rVert_{L_t^\infty L_x^1}\lVert G\rVert_{L_t^\infty L_x^1}\, |\hat{\phi}(2R\nu)|\,|\phi(R^{11/25}\mu)| d\mu d\nu d\tau dt\\
&=\int_X R^{-11/5}|\hat{\phi}(\nu)|\phi(\mu)\lVert F\rVert_{L_t^\infty L_x^1}\lVert G\rVert_{L_t^\infty L_x^1}d\mu d\nu d\tau dt\\
&\lesssim R^{-11/5}|X_1|\,\lVert F\rVert_{L_t^\infty L_x^1}\lVert G\rVert_{L_t^\infty L_x^1}\\
&\lesssim R^{-1/5}\lVert F\rVert_{L_t^\infty L_x^1}\lVert G\rVert_{L_t^\infty L_x^1},
\end{align*}
where to obtain the fourth line we have used $\phi\in C^\infty_c$ and the bound 
\begin{align}
\label{eqaab2}\int_{\mathbb{R}^5} |\hat{\phi}(\nu)|d\nu&\leq C\bigg(\int_{|\nu|\leq 1} \frac{1}{|\nu|^4}d\nu+\int_{|\nu|>1} \frac{1}{|\nu|^6}d\nu\bigg)<\infty,
\end{align}
which follows from the observation that $\phi\in C^\infty_c$ allows us to choose constants $C_m>0$ such that 
\begin{align}
|\hat{\phi}(\xi)|\leq C_m|\xi|^{-m}\label{eq2}
\end{align}
for every $m\geq 1$.

To estimate $(II)$, we write $A_1(t,\tau,\mu,\nu)$ as the sum of two terms and estimate the resulting oscillatory integrals.  More precisely, we obtain 
\begin{align}
\nonumber |(II)|&=\bigg|\int F(t,x)G(\tau,y)\bigg(\int 2^4R^5e^{ix\cdot (\mu+\nu)-iy\cdot (\mu-\nu)} A_1(t,\tau,\mu,\nu)\\
\nonumber &\hspace{0.4in}\cdot(1-\phi(R^{11/25}\mu))\phi(R^{12/25}\nu)d\mu d\nu \bigg)dxdyd\tau dt\bigg|\\
\nonumber &\lesssim \bigg|\int F(t,x)G(\tau,y)\bigg(\int e^{iR\varphi_{1}(\mu,\nu)}\psi(\mu,\nu)d\mu d\nu\bigg) dxdy d\tau dt\bigg|\\
\nonumber &\hspace{0.2in}+\bigg|\int F(t,x)G(\tau,y)\bigg(\int e^{iR\varphi_{-1}(\mu,\nu)}\psi(\mu,\nu)d\mu d\nu\bigg) dxdy d\tau dt\bigg|\\
\nonumber &\lesssim \int_{\mathbb{R}^5\times\mathbb{R}^5\times X_1} |F(t,x)||G(\tau,y)|R^{-74/25}dxdydtd\tau\\
\nonumber &\lesssim R^{-74/25}|X_1|\, \lVert F\rVert_{L_t^\infty L_x^1}\lVert G\rVert_{L_t^\infty L_x^1}\\
&\lesssim R^{-24/25}\lVert F\rVert_{L_t^\infty L_x^1}\lVert G\rVert_{L_t^\infty L_x^1},
\end{align}
where we have used the oscillatory integral estimate
\begin{align}
\label{rltd}\left|\int e^{iR\varphi_\sigma(\mu,\nu)}\psi(\mu,\nu)d\mu d\nu\right| &\lesssim R^{-74/25},
\end{align}
for $\sigma=\pm 1$, with
\begin{align*}
\varphi_\sigma&=\frac{1}{R}[\sigma(t|\mu+\nu|-\tau |\mu-\nu|)+\mu\cdot (x-y)+\nu\cdot (x+y)],\\
\psi&=(1-\phi(R^{11/25}\mu))\phi(R^{12/25}\nu)R^5\\
&\hspace{1.2in}\overline{e^{-ix_0\cdot 2R\nu}\hat{\phi}(2R\nu)}\theta(\mu+\nu)\theta(\mu-\nu)\bigg(1+\frac{(\mu+\nu)\cdot (\mu-\nu)}{|\mu+\nu||\mu-\nu|}\bigg).
\end{align*}

To establish ($\ref{rltd}$), let $(\mu,\nu,\tau,t)$ be a given point in the support of the integrand.  We first show that $\left|\frac{\mu}{|\mu|}\cdot \nabla_\mu \varphi\right|$ is bounded away from zero.  Indeed, using Cauchy-Schwarz followed by the inequality $\sqrt{1-\frac{|\nu|^2}{|\mu+\nu|^2}}\leq \frac{\mu\cdot (\mu\pm \nu)}{|\mu||\mu\pm\nu}$ (which follows from $|\mu|\geq R^{-11/25}$ and $|\nu|\leq 2R^{-12/25}$), we obtain
\begin{align*}
\bigg| \frac{\mu}{|\mu|}\cdot \nabla_{\mu} \varphi\bigg|&=\bigg|\frac{t-\tau}{R}-\frac{t-\tau}{R}+\frac{\mu}{|\mu|}\cdot \bigg[\frac{\sigma t(\mu+\nu)}{R|\mu+\nu|}-\frac{\sigma \tau(\mu-\nu)}{R|\mu-\nu|}+\frac{(x-y)}{R}\bigg]\bigg|\\
&\geq C_2-C_1\max\bigg\{1-\frac{\mu\cdot (\mu+\nu)}{|\mu||\mu+\nu|},1-\frac{\mu\cdot (\mu-\nu)}{|\mu||\mu-\nu|}\bigg\}\\
&\geq C_2-C_1\bigg(1-\sqrt{1-\frac{|\nu|^2}{|\mu+\nu|^2}}\bigg)\\
&\geq C_2-\frac{C_1|\nu|^2}{|\mu+\nu|^2}\\
&\geq C_2-\frac{2C_1R^{-24/25}}{(R^{-11/25}-2R^{-12/25})^2}\\
&\gtrsim 1,
\end{align*}
provided $R$ is chosen sufficiently large.

We now set $a=\left(\frac{\mu}{|\mu|}\cdot \nabla_\mu \varphi\right)^{-1}\frac{\mu}{|\mu|}$.  Observing the identity $e^{iR\varphi}=R^{-6}(a\cdot i\nabla_\mu)^6e^{iR\varphi}$ and integrating by parts, we obtain
\begin{align}
\nonumber \left|\int e^{iR\varphi_\sigma(\mu,\nu)}\psi(\mu,\nu)d\mu d\nu\right|&\lesssim R^{-12/5}\sup_{|\nu|\leq 2R^{-12/25}}\bigg|\int e^{iR\varphi}\psi d\mu\bigg|\\
\nonumber &\lesssim R^{-42/5}\sup_{|\nu|\leq 2R^{-12/25}}\int |(\nabla_\mu\cdot a)^6\psi|d\mu\\
\nonumber &\lesssim R^{-42/5}\int_{|\mu|\geq R^{-11/25}} R^5|\mu|^{-6}d\mu\\
&\lesssim R^{-74/25}
\end{align}
which establishes ($\ref{rltd}$) as desired, where we have used the estimate 
\begin{align*}
\sup_{\nu} |(i\nabla_\mu\cdot a)^6\psi|\leq CR^5|\mu|^{-6}.
\end{align*}

We now turn to the estimate of $(III)$.  To estimate this term, we begin by observing that for $R$ sufficiently large, the conditions $|\mu|>R^{-11/25}$ and $|\nu|<2R^{-12/25}$ give $|\mu|>2|\nu|$.  We therefore obtain  
\begin{align*}
|(III)|&\lesssim R^5\int_{\mathbb{R}^5\times\mathbb{R}^5\times S} \frac{|\nu|^2}{|\mu|^2}\theta(\mu+\nu)|\hat{\phi}(2R\nu)|\lVert F\rVert_{L_t^\infty L_x^1}\lVert G\rVert_{L_t^\infty L_x^1}d\mu d\nu d\tau dt\\
&\lesssim R^5\lVert F\rVert_{L_t^\infty L_x^1}\lVert G\rVert_{L_t^\infty L_x^1}\int_{\mathbb{R}^5\times S}  |\nu|^2|\hat{\phi}(2R\nu)|d\nu d\tau dt\\
&\lesssim R^{-2}\lVert F\rVert_{L_t^\infty L_x^1}\lVert G\rVert_{L_t^\infty L_x^1}|S|\\
&\lesssim R^{-1/26}\lVert F\rVert_{L_t^\infty L_x^1}\lVert G\rVert_{L_t^\infty L_x^1}
\end{align*}
with
\begin{align*}
S:=(X_1\cap \{(t,\tau):|t+\tau|\leq R^{25/26}\}),
\end{align*}
where for the first inequality we have used the bound
\begin{align*}
1-\frac{(\mu+\nu)\cdot (\mu-\nu)}{|\mu+\nu||\mu-\nu|}&\leq \frac{C|\nu|^2}{|\mu|^2}
\end{align*}
on the support of $(1-\phi(R^{11/25}\mu))\phi(R^{12/25}\nu)$,
while for the second and third inequalities we have used the bounds
\begin{align*}
\int_{\mathbb{R}^5} \frac{\theta(\mu+\nu)}{|\mu|^2}d\mu&=\int_{|\mu|<1} |\mu|^{-2}d\mu+\int_{|\mu|\geq 1} \theta(\mu+\nu)d\mu\leq C+\lVert \theta\rVert_{L^1}<\infty
\end{align*}
and
\begin{align*}
\int_{\mathbb{R}^5} |\nu|^2|\hat{\phi}(2R\nu)|d\nu&=R^{-7}\int_{\mathbb{R}^5} |\nu|^2|\hat{\phi}(2\nu)|d\nu\\
&\leq R^{-7}(\lVert \phi\rVert_{L^1}\int_{|\nu|<1} |\nu|^2d\nu+C\int_{|\nu|\geq 1} |2\nu|^{-6}d\nu)\\
&\leq CR^{-7},
\end{align*}
respectively.

To estimate $(IV)$, we proceed similarly to our estimate of $(II)$ above, writing $A_2(t,\tau,\mu,\nu)$ as the sum of two terms and estimating the resulting oscillatory integrals.  More precisely, we obtain 
\begin{align*}
|(IV)|&\lesssim \bigg|\int F(t,x)G(\tau,y)\bigg(\int e^{i\widetilde{\varphi_{1}}(\mu,\nu)}\widetilde{\psi}(\mu,\nu) \\
&\hspace{0.4in}\cdot R^5\overline{e^{-ix_0\cdot 2R\nu}\hat{\phi}(2R\nu)}\phi(R^{12/25}\nu)d\mu d\nu\bigg) dxdy d\tau dt\bigg|\\
&\hspace{0.2in}+\bigg|\int F(t,x)G(\tau,y)\bigg(\int e^{i\widetilde{\varphi_{-1}}(\mu,\nu)}\widetilde{\psi}(\mu,\nu)\\
&\hspace{0.4in}\cdot R^5\overline{e^{-ix_0\cdot 2R\nu}\hat{\phi}(2R\nu)}\phi(R^{12/25}\nu)d\mu d\nu\bigg) dxdy d\tau dt\bigg|\\
&\lesssim \int_{\mathbb{R}^5\times \mathbb{R}^5\times X_1} |F(t,x)||G(\tau,y)|R^{-129/52}dxdydtd\tau\\
&\lesssim \lVert F\rVert_{L_t^\infty L_x^1}\lVert G\rVert_{L_t^\infty L_x^1}R^{-129/52}|X_1|\\
&\lesssim \lVert F\rVert_{L_t^\infty L_x^1}\lVert G\rVert_{L_t^\infty L_x^1}R^{-25/52}.
\end{align*}
where we have used the oscillatory integral estimate
\begin{align}
\label{le2e}
\int \bigg(\int e^{i\widetilde{\varphi_\sigma}}\widetilde{\psi}d\mu\bigg) R^5\overline{e^{-ix_0\cdot 2R\nu}\hat{\phi}(2R\nu)}\phi(R^{12/25}\nu)d\nu\lesssim R^{-129/52}
\end{align}
for $\sigma\in \{-1,1\}$ and
\begin{align*}
\widetilde{\varphi_\sigma}&=\sigma(t|\mu+\nu|+\tau|\mu-\nu|)+\mu(x-y)+\nu(x+y),\\
\widetilde{\psi}&=\theta(\mu+\nu)\theta(\mu-\nu)\bigg(1-\frac{(\mu+\nu)\cdot (\mu-\nu)}{|\mu+\nu||\mu-\nu|}\bigg)(1-\phi(R^{11/25}\mu)).
\end{align*}

To obtain $(\ref{le2e})$, we write the left hand side of ($\ref{le2e}$) as
\begin{align*}
\sum_{j=1}^{10} \int \bigg(\int e^{i\widetilde{\varphi}_\sigma}\widetilde{\psi}\eta_j(\mu)d\mu\bigg)R^5\overline{e^{-ix_0\cdot 2R\nu}\hat{\phi}(2R\nu)}\phi(R^{12/25}\nu)d\nu.
\end{align*}
where the functions $\eta_j$ form a partition of unity subordinate to the open cover
\begin{align*}
\{\max \{|\mu_i|:i\neq k\}<\frac{25}{12}\mu_k\},\quad \{\max \{|\mu_i|:i\neq k\}<-\frac{25}{12}\mu_k\}
\end{align*}
for $k=1,\cdots 5$.

We consider the $j=1$ term and note that by symmetry one may estimate the other terms in a similar manner.  Invoking the Van der Corput lemma, we bound the modulus of the left hand side of ($\ref{le2e}$) by 
\begin{align*}
&\hspace{0.2in}\int R^5|\hat{\phi}(2R\nu)|\,|\phi(R^{12/25}\nu)|\, \left|\int \left(\int e^{i\widetilde{\varphi}_\sigma}\widetilde{\psi}\eta_1d\mu_1\right)d\mu_2\cdots d\mu_5\right|d\nu\\
&\hspace{0.2in}\lesssim \int_{\mathbb{R}^5}\int_{Y\times [0,\infty)} \left(\frac{\mu_5}{|t+\tau|}\right)^{1/2} R^5|\hat{\phi}(2R\nu)|\,|\phi(R^{12/25}\nu)| \\
&\hspace{1in}\cdot \left(\lVert \widetilde{\psi}\rVert_{L_{\mu_1}^\infty}+\lVert \partial_{\mu_1}\widetilde{\psi}\rVert_{L_{\mu_1}^1}\right)d\mu_2\cdots d\mu_5d\nu\\
&\hspace{0.2in}\lesssim \int_{\mathbb{R}^5}\int_{Y\times [0,\infty)} \left(\frac{|\nu|^2}{|t+\tau|^{1/2}}\right)R^5|\hat{\phi}(2R\nu)|\,|\phi(R^{12/25}\nu)|\\
&\hspace{1in}\left(\frac{1}{\langle \mu_5\rangle^8}(2\mu_5^{3/2}+\mu_5^{5/2})\right)d\mu_2\cdots d\mu_5d\nu\\
&\hspace{0.2in}\lesssim R^{-129/52},
\end{align*}
with $Y=\{(\mu_2,\mu_3,\mu_4):|\mu_i|\leq \mu_5\}$, where to obtain the second inequality we have observed that for every $(\mu,\nu)$ in the support of the integrand and $(t,\tau)\in X_1$, we have the inequalities $|\mu|^2-|\mu_1|^2\sim |\mu|^2$ and $|\mu|\sim\mu_5$, as well as
\begin{align*}
\bigg|\partial_{\mu_1}^2\widetilde{\varphi}_\sigma-\sigma\cdot (t+\tau)\frac{|\mu|^2-\mu_1^2}{|\mu|^3}\bigg|\leq \frac{CR^{1-1/25}}{|\mu|}.
\end{align*}
and for $R$ sufficiently large,
\begin{align*}
|\partial_{\mu_1}^2\widetilde{\varphi}_\sigma|&\geq |(t+\tau)\frac{|\mu|^2-\mu_1^2}{|\mu|^3}|-\bigg|\partial_{\mu_1}^2\widetilde{\varphi}_\sigma-\sigma\cdot (t+\tau)\frac{|\mu|^2-\mu_1^2}{|\mu|^3}\bigg|\\
&\geq |(t+\tau)\frac{|\mu|^2-\mu_1^2}{|\mu|^3}|-\frac{CR^{1-1/25}}{|\mu|}\\
&\geq \frac{|t+\tau|}{\mu_5}-\frac{CR^{-1/650}|t+\tau|}{\mu_5}\\
&\geq \frac{|t+\tau|}{2\mu_5},
\end{align*}
and to obtain the last inequality we have recalled that $\frac{1}{|t+\tau|^{1/2}}\leq \frac{1}{R^{25/52}}$ on the domain of integration.

It remains to estimate $(V)$, for which we proceed as in $(I)$.  In particular, we observe that by ($\ref{eq2}$), we have the inequality $|\hat{\phi}(2R\nu)|\leq C|2R\nu|^{-9}$, from which we obtain the bound
\begin{align*}
|(V)|&\lesssim \int_X R^5|A(t,\tau,\mu,\nu)|(1-\phi(R^{11/25}\mu))(1-\phi(R^{12/25}\nu))\\
&\hspace{0.4in}\lVert F\rVert_{L_t^\infty L_x^1}\lVert G\rVert_{L_t^\infty L_x^1}d\mu d\nu d\tau dt\\
&\lesssim \int_X R^5\theta(\mu+\nu)|\hat{\phi}(2R\nu)|(1-\phi(R^{12/25}\nu))\lVert F\rVert_{L_t^\infty L_x^1}\lVert G\rVert_{L_t^\infty L_x^1}d\mu d\nu d\tau dt\\
&\lesssim \lVert \theta\rVert_{L^1}\lVert F\rVert_{L_t^\infty L_x^1}\lVert G\rVert_{L_t^\infty L_x^1}\int_{X\cap \{\nu>R^{-12/25}\}} R^{-4}|\nu|^{-9}d\nu d\tau dt\\
&\lesssim \lVert \theta\rVert_{L^1}\lVert F\rVert_{L_t^\infty L_x^1}\lVert G\rVert_{L_t^\infty L_x^1}R^{-52/25}|X_1|\\
&\lesssim \lVert F\rVert_{L_t^\infty L_x^1}\lVert G\rVert_{L_t^\infty L_x^1}R^{-2/25}.
\end{align*}
\end{proof}

\subsection*{Acknowledgements}

The author would like to thank M. Visan for bringing this problem to our attention and for comments on an earlier version of this manuscript, as well as W. Beckner and N. Pavlovi\'c for useful conversations.  This material is based upon work supported by the National Science Foundation under agreement Nos. DMS-0635607 and DMS-0808042.  Any opinions, finding and conclusions or recommendations expressed in this material are those of the author and do not necessarily reflect the views of the National Science Foundation.


\begin{thebibliography}{9}
\bibitem{BahouriGerard} H. Bahouri and P. G\'erard, \textit{High frequency approximation of solutions to critical nonlinear wave equations.} Amer. J. Math.  121 (1999), 131--175.
\bibitem{BulutMax} A. Bulut, \textit{Maximizers for the Strichartz inequalities for the Wave Equation.} Diff. Int. Eq. \textbf{23} (2010), 1035--1072.
\bibitem{BulutCubic} A. Bulut, \textit{Global Well-posedness and Scattering for the Defocusing Energy-Supercritical Cubic Nonlinear Wave Equation.} J. Funct. Anal. 263 (2012), no. 6, 1609-1660.
\bibitem{BulutRadial} A. Bulut, \textit{The radial defocusing energy-supercritical cubic nonlinear wave equation.} Preprint (2011), arXiv:1104.2002.
\bibitem{BulutThesis} A. Bulut, \textit{Global Well-posedness and Scattering for the Defocusing Energy-Supercritical Cubic Nonlinear Wave Equation.} Ph.D. Thesis, 2011, The University of Texas at Austin.
\bibitem{BulutContemporary} A. Bulut, \textit{The defocusing cubic nonlinear wave equation in the energy-supercritical regime.} Recent Advances in Harmonic Analysis and Partial Differential Equations, 12-22, Contemp. Math. 581, Amer. Math. Soc., Providence, RI, 2012.
\bibitem{CKSTT} J. Colliander, M. Keel, G. Staffilani, H. Takaoka and T. Tao, \textit{Global well-posedness and scattering for the energy-critical nonlinear Schr\"odinger equation in $\mathbb{R}^3$.} Ann. Math. 167 (2008), 767--865.
\bibitem{ESS} L. Escauriaza, G. Seregin and V. Sver\'ak, \textit{$L^{3,\infty}$-solutions of Navier-Stokes equations and backward uniqueness.} Russ. Math. Surv. 58, 2 (2003), 211--250.
\bibitem{GinibreVelo} J. Ginibre and G. Velo, \textit{Generalized Strichartz Inequalities for the Wave Equation.} J. Funct. Anal. 133, 1 (1995), 50--68.
\bibitem{Grillakis1} M. Grillakis, \textit{Regularity and asymptotic behaviour of the wave equation with a critical nonlinearity.} Ann. of Math. 132 (1990), 485--509.
\bibitem{Grillakis2} M. Grillakis, \textit{Regularity for the wave equation with a critical nonlinearity.} Comm. Pure Appl. Math. 45 (1992), 747--774.
\bibitem{Kapitanskii}  L. Kapitanskii, \textit{The Cauchy problem for the semilinear wave equation. I., II., III.} Zap. Nauchn. Sem. Leningrad. Otdel. Mat. Inst. Steklov. (LOMI) 163 (1987), 76-104, 182 (1990), 38-85, and 181 (1990), 24--64.
\bibitem{KeelTao} M. Keel and T. Tao, \textit{Endpoint Strichartz Estimates.} Amer. J. Math. 120, 5 (1998) 955--980.
\bibitem{KMECriticalNLS} C. Kenig and F. Merle, \textit{Global well-posedness, scattering and blow-up for the energy critical, focusing, non-linear Schr\"odinger equation in the radial case.} Invent. Math. 166 (2006), 645--675.
\bibitem{KMECriticalNLW} C. Kenig and F. Merle, \textit{Global well-posedness, scattering and blow-up for the energy critical focusing non-linear wave equation.} Acta Math. 201 (2008), 147-212.
\bibitem{KMESupercriticalNLW} C. Kenig and F. Merle, \textit{Nondispersive radial solutions to energy supercritical non-linear wave equations, with applications.} Amer. J. Math. 133 (2011), no. 4, 1029--1065.
\bibitem{KModd}  C. Kenig and F. Merle, \textit{Radial solutions to energy supercritical wave equations in odd dimensions.} Disc. Cont. Dyn. Sys. A, 4 (2011) 1365--1381.
\bibitem{KillipTaoVisan} R. Killip, T. Tao and M. Visan, \textit{The cubic nonlinear Schr\"odinger equation in two dimensions with radial data.} J. Eur. Math. Soc. 11 (2009) 1203--1258.
\bibitem{KVECriticalNLS} R. Killip and M. Visan, \textit{The focusing energy-critical nonlinear Schr\"odinger equation in dimensions five and higher.}  Amer. J. Math. 132 (2009), pp. 361--424.
\bibitem{KVESupercriticalNLS} R. Killip and M. Visan, \textit{Energy-supercritical NLS: critical $\dot{H}^s$-bounds imply scattering.}  Comm. Par. Diff. Eq. 35 (2010), 945--987.
\bibitem{KVNLW3} R. Killip and M. Visan, \textit{The defocusing energy-supercritical nonlinear wave equation in three space dimensions.} Trans. Amer. Math. Soc. 363 (2011), 3893--3934.
\bibitem{KVNLWradial} R. Killip and M. Visan, \textit{The radial defocusing energy-supercritical nonlinear wave equation in all space dimensions.} Proc. Amer. Math. Soc. 139 (2011), 1805--1817.
\bibitem{Nakanishi} K. Nakanishi, Scattering theory for the nonlinear Klein-Gordon equation with Sobolev critical power. Internat. Math. Res. Notices  1999,  no. 1, 31--60. 
\bibitem{Pecher} H. Pecher, Nonlinear small data scattering for the wave and Klein-Gordon equation. Math. Z.  185  (1984), no. 2, 261--270.
\bibitem{Rauch} J. Rauch, I. The $u^5$ Klein-Gordon equation. II. Anomalous singularities for semilinear wave equations. Nonlinear partial differential equations and their applications. Coll`ege de France Seminar, Vol. I (Paris, 1978/1979), pp. 335-364, Res. Notes in Math., 53, Pitman, Boston, Mass.-London, 1981.
\bibitem{ShatahStruwe1} J. Shatah and M. Struwe, \textit{Regularity results for nonlinear wave equations.} Ann. of Math. 138 (1993), 503--518.
\bibitem{ShatahStruwe2} J. Shatah and M. Struwe, \textit{Well-posedness in the energy space for semilinear wave equations with critical growth.} Internat. Math. Res. Notices  1994,  no. 7, 303ff., approx. 7 pp. (electronic).
\bibitem{ShatahStruwe} J. Shatah and M. Struwe, Geometric wave equations. Courant Lecture Notes in Mathematics, 2. New York University, Courant Institute of Mathematical Sciences, New York; American Mathematical Society, Providence, RI, 1998. viii+153 pp.
\bibitem{Struwe} M. Struwe, \textit{Globally regular solutions to the $u^5$ Klein-Gordon equation.} Ann. Scuola Norm. Sup. Pisa Cl. Sci. 15 (1988), 495-513 (1989).
\bibitem{Tao1} T. Tao, Spacetime bounds for the energy-critical nonlinear wave equation in three spatial dimensions. Dyn. Partial Differ. Equ.  3  (2006), no. 2, 93--110. 
\bibitem{TaoBook} T. Tao, Nonlinear dispersive equations: local and global analysis.  CBMS regional conference series in Mathematics, July 2006.
\bibitem{TaoVisanZhang} T. Tao, M. Visan and X. Zhang, \textit{Minimal-mass blowup solutions of the mass-critical NLS.} Forum Math.  20  (2008), no. 5, 881--919.
\end{thebibliography}
\end{document}